\documentclass[a4paper, reqno, 12pt]{amsart}
\usepackage{a4wide}
\usepackage[english]{babel}
\usepackage[normalem]{ ulem }
\usepackage{soul}
\usepackage{graphicx}

\usepackage[T1]{fontenc}
\usepackage[utf8]{inputenc}
\usepackage{amsmath,amsfonts,amssymb,amsopn,amscd,amsthm, mathrsfs}

\usepackage{gensymb}
\usepackage{comment}
\usepackage{dsfont}
\usepackage{bm}
\usepackage{graphicx}
\usepackage{xcolor}
\usepackage[colorlinks]{hyperref}
\usepackage{epigraph}
\usepackage{enumerate}
\usepackage{xargs}

\usepackage{marginnote}

\title[ To spike or not to spike: the whims of the Wonham filter ]
      { To spike or not to spike: the whims of the Wonham filter in the strong noise regime}

\author[Bernardin]{\textsc{C\'edric Bernardin}} 
\address{Faculty of Mathematics, National Research University Higher School of Economics -- 6 Usacheva, 119048 Moscow, Russia}
\email{{\tt sedric.bernardin@gmail.com}}

\author[Chhaibi]{\textsc{Reda Chhaibi}}
\address{Institut de math\'ematiques, UMR5219, Universit\'e de Toulouse, CNRS, UPS, F-31062 Toulouse Cedex 9, France}
\email{{\tt reda.chhaibi@math.univ-toulouse.fr}}

\author[Najnudel]{\textsc{Joseph Najnudel}} 
\address{University of Bristol, Beacon House, Queens Road, Bristol, BS8 1QU, UK}
\email{{\tt joseph.najnudel@bristol.ac.uk}}

\author[Pellegrini]{\textsc{Cl\'ement Pellegrini}}
\address{Institut de math\'ematiques, UMR5219, Universit\'e de Toulouse, CNRS, UPS, F-31062 Toulouse Cedex 9, France}
\email{{\tt clement.pellegrini@math.univ-toulouse.fr}}

\date{\today}

\allowdisplaybreaks[4]

\def\half{\frac{1}{2}}

\def\B{{\mathbb B}}

\def\Q{{\mathbb Q}}
\def\R{{\mathbb R}}

\def\H{{\mathbb H}}
\def\S{{\mathbb S}}

\def\X{\mathbb{X}}
\def\L{\mathbb{L}}

\def\P{{\mathbb P}}
\def\E{{\mathbb E}}

\def\X{{\mathbb X}}

\def\Ec{{\mathcal E}}
\def\Fc{{\mathcal F}}
\def\Gc{{\mathcal G}}

\def\Lc{{\mathcal L}}

\def\Oc{{\mathcal O}}
\def\Pc{{{\mathcal P}}}

\def\xb{{\mathbf x}}
\def\yb{{\mathbf y}}

\newtheorem{thm}{Theorem}[section]
\newtheorem{proposition}[thm]{Proposition}

\newtheorem{question}[thm]{Question}

\newtheorem{lemma}[thm]{Lemma}

\newtheorem{rmk}[thm]{Remark}

\numberwithin{equation}{section}
\numberwithin{figure}{section}

\renewcommand{\imath}{i}

\let\oldtocsection=\tocsection
 
\let\oldtocsubsection=\tocsubsection
 
\let\oldtocsubsubsection=\tocsubsubsection
 
\renewcommand{\tocsection}[2]{\hspace{0em}\oldtocsection{#1}{#2}}
\renewcommand{\tocsubsection}[2]{\hspace{2em}\oldtocsubsection{#1}{#2}}
\renewcommand{\tocsubsubsection}[2]{\hspace{4em}\oldtocsubsubsection{#1}{#2}}

\begin{document}

\begin{abstract}

We study the celebrated Shiryaev-Wonham filter \cite{wonham1964} in its historical setup  where the hidden Markov jump process has two states. We are interested in the weak noise regime for the observation equation. Interestingly, this becomes a strong noise regime for the filtering equations.

Earlier results of the authors show the appearance of spikes in the filtered process, akin to a metastability phenomenon. This paper is aimed at understanding the smoothed optimal filter, which is relevant for any system with feedback. In particular, we exhibit a sharp phase transition between a spiking regime and a regime with perfect smoothing.
\end{abstract}

\medskip

\renewcommand{\subjclassname}{%
  \textup{2010} Mathematics Subject Classification}

\maketitle

\hrule
\tableofcontents
\hrule

\section{Introduction}

Filtering Theory adresses the problem of estimating a \textit{hidden process} $\xb = (\xb_t \; ; \; t\ge 0)$ which can not be directly observed. At hand, one has access to an \textit{observation process} which is naturally correlated to $\xb$. The most simple setup, called the ``signal plus noise'' model,  is the one where the observation process $\yb^\gamma = (\yb^\gamma_t \; ; \; t\ge 0)$ is of the form
\begin{align}
\label{eq:observation}
d\yb^\gamma_t = \xb_t dt+\cfrac{1}{\sqrt{\gamma}} \ dB_t 
\end{align}
where $B = (B_t \; ; \; t\ge 0)$ is a standard Wiener process and $\gamma>0$.  Moreover it is natural to assume that the noise is intrinsic to the observation system, so that the Brownian motion $B = B^\gamma$ has no reason of being the same for different values of $\gamma$. See Figure \ref{fig:processes_xy} for an illustration which visually highlights the difficulty of recognizing a drift despite Brownian motion fluctuations. In this paper we shall focus on the case where $(\textbf x_t \; ; \; t\ge 0)$ is a pure jump Markov process on $\{0,1\}$ with c\`adl\`ag trajectories. We denote $\lambda p$ (resp. $\lambda (1-p)$) the jump rate between $0$ and $1$ (resp. between $1$ and $0$), with $p\in (0,1)$ and $\lambda>0$.  This is the historical setting of the celebrated Wonham filter \cite[Eq. (19)]{wonham1964}.

In the mean square sense, the best estimator taking value in $\{0,1\}$ at time $t$ of $\textbf x_t$, given the observation $({\textbf y}^\gamma_s)_{s\leq t}$, is equal to 
\begin{align}
\label{eq:hatx}
{\hat \xb^{\gamma}}_t = & \ \mathds{1}_{\left\{ \pi^\gamma_t > \half \right\} }
\end{align}
where $\pi^\gamma_t$ is the conditional probability 
\begin{align}
\label{def:pi}
\pi^\gamma_t := & \ \P\left( {\bf x}_t = 1 \ | \ \left({\bf y}^\gamma_s\right)_{s\le t} \right) \ .
\end{align}

Our interest lies in the situation where the intensity $1/\sqrt{\gamma}$ of the observation noise is small, i.e. $\gamma$ is large. At first glance, one could argue that weak noise limits for the observation process are not that interesting because we are dealing with extremely reliable systems since they are subject to very little noise. As such, one would naively expect that observing $\yb^\gamma$ allows an optimal recovery of $\xb$ as $\gamma \rightarrow \infty$, via a straightforward and stable manner. This paper aims at demonstrating that this regime is more surprizing and interesting from both a theoretical and a practical point of view.  

\begin{figure}[htp]
\includegraphics[scale=0.5]{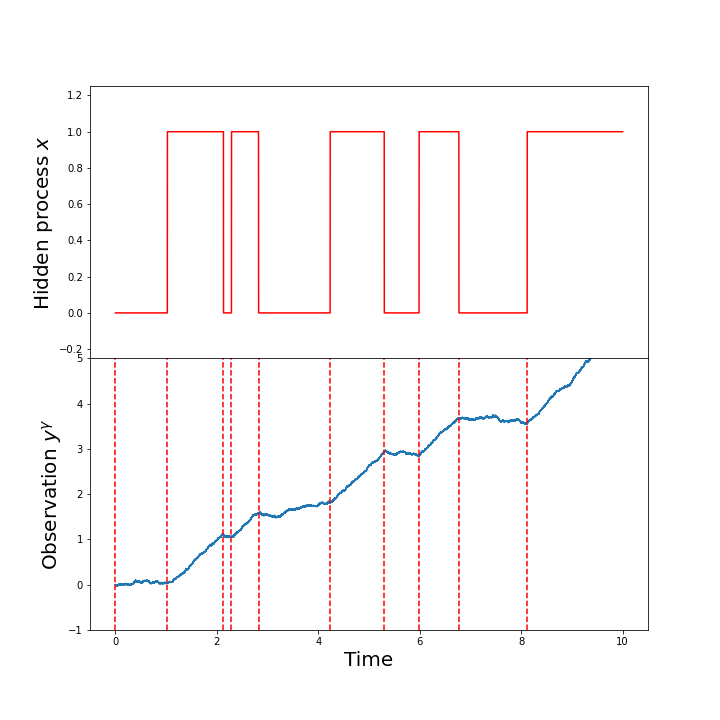}
\caption{Numerical simulation of the hidden process $\xb$ and the observation process $\yb^\gamma$ for $\gamma = 10^2$. The challenge is to infer the drift of $\yb^\gamma$, in spite of Brownian noise and in a very short window. Parameters are $\lambda=1.3$ and $p=0.4$. There are $10^6$ time steps to discretize $[0,10]$. The code is available at the online repository
\quad \quad \quad \quad \quad \quad \quad \quad \quad \quad \quad \quad \quad \quad \quad \quad
\quad \quad \quad \quad \quad \quad \quad \quad \quad \quad \quad \quad 
\url{https://github.com/redachhaibi/Spikes-in-Classical-Filtering}}
\label{fig:processes_xy}
\end{figure}

\medskip

{\bf A motivating example. } Let us describe a simple situation that falls into that scope and motivates our study. Consider for example a single classical bit -- say, inside of a DRAM chip. The value of the bit is subject to changes, some of which are caused by CPU instructions and computations, some of which are due to errors. The literature points to spontaneous errors due to radiation, heat and various conditions \cite{dram_study}. The value of that process is modeled by the Markov process $\xb $ as defined above.  Here, the process ${\bf y}^\gamma$ is the electric current received by a sensor on the chip, which monitors any changes. Any retroaction, for example code correction in ECC memory \cite{ecc_study1, ecc_study2}, requires the observation during a finite window $\delta>0$. And the reaction is at best instantaneous. For anything meaningful to happen, everything depends thus on the behavior of:
\begin{align}
\label{eq:def_process}
\pi_t^{\delta, \gamma} := & \ \P\left( {\bf x}_{t-\delta} = 1 \ | \ \left({\bf y}^\gamma_s\right)_{s\le t}  \right) \ ,
\end{align}
and instead to consider the estimator ${\bf{\hat x}}_t^\gamma$ given by Eq. \eqref{eq:hatx}, we are left with the estimator
\begin{equation*}
\label{eq:hatxdelta}
{\hat \xb^{\delta, \gamma}}_t = \mathds{1}_{\left\{ \pi^{\delta,\gamma}_t > \half \right\} } \ .
\end{equation*}
From an engineering point of view, it is the interplay between different time scales which is important in order to design a system with high performance: if the noise is weak, how fast can a feed-back response be? For a given process ${\mathbf z}=({\mathbf z}_{t}\; ; t\ge 0)$ with values in $[0,1]$ we denote the hitting time of $(\half,1)$ by $T ({\mathbf z}):=\inf \left\{ t\ge 0\; ; \; {\mathbf z}_t > \half \right\}$. Assume for example that initially ${\bf x}_0=0$. For a given time $t>0$, a natural problem is to estimate, as $\gamma \to \infty$, the probability to predict a false value of the bit given its value remains equal to $0$ during the time interval $[0,t]$, i.e.
\begin{equation}
\label{eq:ptt}
{\mathbb P} \left( T ({\hat \xb}^{\delta,\gamma}) \le t \;  \vert \; T ({\bf x}) >t \right) \ .
\end{equation}

\subsection{Informal statement of the result.}
A consequence of the results of this paper is the precise identification of the regimes  $\delta:= \delta(\gamma)$ for which the probability in \eqref{eq:ptt} vanishes or not as $\gamma \to \infty$:\\

\begin{itemize}
\item If $\limsup_{\gamma\to \infty} \delta(\gamma) \tfrac{\gamma}{\log \gamma} <2$, i.e. $\delta$ is too small,  the retroaction/control system can be surprised by a so--called spike, causing a misfire in detecting the regime change and the limiting error probability in Eq. \eqref{eq:ptt} is equal to $1-\exp \left( -\lambda p t \right)$;\\
\item If $\liminf_{\gamma\to \infty} \delta(\gamma) \tfrac{\gamma}{\log \gamma} >8$, i.e. $\delta$ is sufficiently large, the estimator will be very good at detecting jumps of the Markov process ${\bf x}$, the limiting error probability in Eq. \eqref{eq:ptt} vanishing. However the reaction time will deteriorate.
\end{itemize}

\bigskip
In the previous statement the presence of the number $8$ is due to a technical estimate and it is almost clear that it could be replaced by $2$. We refer the reader to Section \ref{subsection:rmk:Cge2} and Figure \ref{fig:filter_smoothed}, even if the numerical simulations are not totally convincing. In this remark lies the only estimate which limits the extension of the claim to the case $C>2$. Hence, there is no doubt that the transition occurs for $\delta(\gamma) =2\  \frac{\gamma}{\log \gamma}$. 

\bigskip

While the literature usually focuses on $L^2$ considerations for filtering processes, we focus on this article on pathwise properties of the filtering process under investigation when $\gamma \rightarrow \infty$. Indeed, it is clear that the question addressed just above cannot be answered in an $L^2$ framework only.

\bigskip

Let us now present in some informal way the reasons for which we have this difference of behavior. As it will be recalled later the process $\pi^\gamma = \left( \pi^\gamma_t \ ; \ t \geq 0 \right)$ satisfies in law
\begin{equation}
\label{eq:sde}
d\pi^\gamma_t = - \lambda \left( \pi^\gamma_t - p \right) dt +  \sqrt{\gamma}  \pi^\gamma_t \left( 1 - \pi^\gamma_t \right) dW_t \ ,
\end{equation}
where $W=(W_t \; ; \; t\ge 0)$ is a Brownian motion with a \textit{now strong parameter} $\sqrt{\gamma}$ in front of it. This is the so called Shiryaev-Wonham filtering theory \cite{wonham1964, LS01, vanhandel}. As shown in \cite{bernardin2018spiking}, when $\gamma$ goes to infinity the process $\pi^\gamma$ converges in law to an unusual and singular process in a suitable topology (see Figure \ref{fig:filter}).  Indeed as exhibited in the figure, the limiting process is the Markov jump process $({\bf x}_t \; ; \; t\ge 0)$ but decorated with vertical lines, called spikes, whose extremities are distributed according an inhomogeneous point Poisson process. As we can observe on Figure \ref{fig:filter_smoothed}, if $\delta$ is sufficiently large, the spikes in the process $\pi^{\gamma, \delta}$ are suppressed while if $\delta$ is sufficiently small they survive. The spikes are responsible of the non vanishing error probability in Eq. \eqref{eq:ptt} since they are interpreted by the estimator ${\hat \xb}^{\delta, \gamma}$ as a jump from $0$ to $1$ of the process $\bf x$. The fact that the transition between the two regimes is precisely $2 \frac{\log \gamma}{\gamma}$ is more complicated to explain without going into computational details. Building on our earlier results, we examine hence in this paper the effect of smoothing and the relevance of various time scales required for filtering, smoothing and control in the design of a system with feedback.


\begin{figure}[htp]
\includegraphics[scale=0.5]{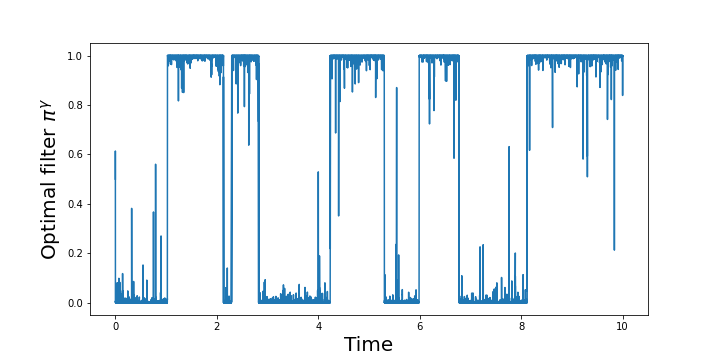}
\caption{``The whims of the Wonham filter'': Informally, on a very short time interval, it is difficult to distinguish between a change in the drift of $\yb^\gamma$ and an exceptionnal time of Brownian motion. The figure shows a numerical simulation of the process $\left( \pi_t^\gamma \ ; \ t \geq 0 \right)$ for the same realization of $\xb$ as Fig. \ref{fig:processes_xy}. Same time discretization. This time we chose the larger $\gamma=10^4$ to highlight spikes.}
\label{fig:filter}
\end{figure}

\begin{figure}[htp]
\includegraphics[scale=0.5]{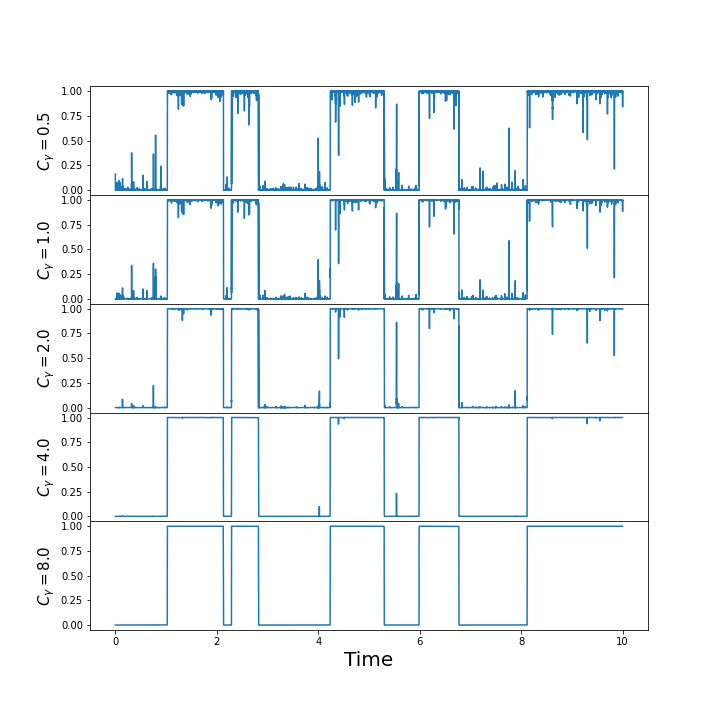}
\caption{Numerical simulation of the process $( \pi_t^{\delta, \gamma} \ ; \ t \geq 0)$ for the same realisation of $\xb$ as Fig. \ref{fig:processes_xy}. Same time discretisation. We have $\gamma=10^4$ and $\delta_\gamma = C \frac{\log \gamma}{\gamma}$, with $C \in \{ \half , 1, 2, 4, 8 \}$. }
\label{fig:filter_smoothed}
\end{figure}

\medskip


\begin{rmk}[Duality between weak and strong noise]
Notice that the observation equation \eqref{eq:observation} has a factor $\frac{1}{\sqrt{\gamma}}$, while the filtering equation \eqref{eq:sde} has a factor $\sqrt{\gamma}$. This is a well-known duality between the weak noise limit in the observation process and the strong noise limit in the filtered state.

In fact, when analyzing the derivation of the Wonham-Shiryaev filter, this is simply due to writing:
$$ d{\bf y}^\gamma_t 
   = \frac{1}{\sqrt{\gamma}} \left( dB_t + \sqrt{\gamma} \, \xb_t \, dt \right) 
   =: \frac{1}{\sqrt{\gamma}} dW^\Q_t \ ,$$
and using the Girsanov transform to construct a new measure $\Q$, for the Kallianpur-Streibel formula, under which $W^\Q$ is a Brownian motion -- \cite[Chapter 7]{vanhandel}.
\end{rmk}

\subsection{Literature review of filtering theory in the \texorpdfstring{$\gamma \rightarrow \infty$}{strong noise} regime.}
The understanding of the behavior of the classical filter for jump Markov processes with small Brownian observation noise has attracted some attention in the 90's. Most of the work there is focused on the long time regime \cite{wonham1964,Khasminski-LazarevaAoS92, Khasminskii-ZeitouniSPA,Atar-ZeitouniSIAM97,Atar-ZeitouniAIHP97,assaf1997estimating}, by studying for example stationary measures, asymptotic stability or transmission rates. In the case where the jump Markov process is replaced by a diffusion process with a signal noise, possibly small, \cite{PicardSIAM,Atar-Zeitouni-SCL98} study the efficiency (in the $L^2$ sense and at fixed time) of some asymptotically optimal filters. In \cite{Pardoux-Zeitouni05} are obtained quenched large deviations principles for the distribution of the optimal filter at a fixed time for one dimensional nonlinear filtering in the small observation noise regime -- see also \cite{Sumith-AmitSIAM22}.  In a similar context Atar obtains in \cite{AtarAoP98} some non-optimal upper bounds for the asymptotic rate of stability of the filter. 

Going through the aforementioned literature one can observe that the term $\log \gamma /\gamma$ already appears in those references. Indeed the quantities of interest include the (average) long time error rate \cite[Eq. (1.4)]{assaf1997estimating}
$$ \alpha^*
   =
   \lim_{t\to\infty}\frac{1}{t}\int_0^t\min(\pi_s^\gamma,1-\pi_s^\gamma) \, ds$$
or the probability of error in long time (\cite{wonham1964} and \cite[Theorem 1']{Khasminskii-ZeitouniSPA})
$$ \mathcal P_{err}(\gamma)
   = \lim_{t\to\infty}\inf_{\zeta\in L_{\infty}(\mathcal F_t^{\yb })}\mathbb P(\zeta\neq\xb_t)
   = \lim_{t\to\infty} \P(\hat\xb_t\neq\xb_t)
$$
or the long time mean squared error \cite{golubev2000}
$$ \Ec_{mse}(\gamma)
   = \lim_{t\to\infty}\inf_{\zeta\in L_{\infty}(\mathcal F_t^{\yb })}\mathbb E(\zeta-\xb_t)^2
   = \lim_{t\to\infty}\mathbb E(\pi_t^\gamma-\xb_t)^2 \ .$$ 
Here $\Fc^{\yb}$ denotes the natural filtration of $\yb = \yb^\gamma$. These quantities are shown to be of order $\tfrac{\log \gamma}{\gamma}$ up to a constant which is related to the invariant measure of $\xb$ and some relative entropy but which is definitively not $2$ -- see \cite[Eq. (3)]{golubev2000}. Note that all these quantities are of asymptotic nature and their analysis goes through the invariant measure. Beyond the appearance of the quantity $\tfrac{\log \gamma}{\gamma}$, which is fortuitous, our results are of a completely different nature since we want to obtain a sharp result on a fixed finite time interval. Also, due to the spiking phenomenon and the singularity of the involved processes, there is no chance that the limits can be exchanged.

To the best of the authors' knowledge, this paper is the first of its kind to aim for a trajectorial description of the limit, in the context of classical filtering theory. However, the spiking phenomenon has first been identified in the context of quantum filtering \cite[Fig. 2]{mabuchi2009} and more specifically, for the control and error correction of qubits. The spiking phenomenon is already seen as a possible source of error where correction can be made while no error has occurred. To quote \cite[Section 4]{mabuchi2009}, when discussing the relevance of the optimal Wonham filter in the strong noise regime, it ``is not a good measure of the information
content of the system, as it is very sensitive to the whims of the filter''. 

Then, in the studies of quantum trajectories{\footnote{Mathematically speaking quantum trajectories are (multi)-dimensional diffusion processes with a special form of the drift and volatility.}} with strong measurement, a flurry of developments have recently taken place, following the pioneering works of Bauer, Bernard and Tilloy \cite{tilloy2015spikes, bauer2016zooming}. Strong interaction with the environment, which is natural in the quantum setting, corresponds to a strong noise in the quantum trajectories.

Note that, the SDEs are the same when comparing classical to quantum filtering. Nevertheless, the noise has a fundamentally different nature. And there is no hidden process $\xb$ in the quantum setting. See \cite{benoist2021emergence, bernardin2018spiking} for a recent account and more references on the quantum literature.


\section{Statement of the problem and Main Theorem}

\subsection{The Shiryaev-Wonham filter}
Let us start by presenting the Shiryaev-Wonham filter and refer to \cite{wonham1964, LS01, vanhandel} for more extensive material. 

\subsubsection{General setup} In this paragraph only, we present the Shiryaev-Wonham  filter on $n$ states, which will allow to highlight the structural aspects of Eq. \eqref{eq:sde}. In general, one considers a Markov process $\xb = \left( \xb_t \ ; \ t \geq 0 \right)$ on a finite state space $E=\{x_1, x_2, \dots, x_n\}$, of cardinal $n\ge 2$, and a continuous observation process ${\bf y}^\gamma$ of the usual additive form ``signal plus noise'':
\begin{align*}
\label{eq:general_observation}
 d{\bf y}^\gamma_t := & \ G\left( \xb_t \right) dt + \frac{1}{\sqrt{\gamma}} \,  dB_t \ .
\end{align*}
Here $G: E \rightarrow \R$ is a function taking distinct values for identifiability purposes. The filtered state is given by:
$$ \rho^\gamma_t (x_i)
   :=
   \P\left( \xb_t = x_i \ \vert \ \left({\bf y}^\gamma_s\right)_{s\le t} \right) \ .
$$
The generator of $\xb$ is denoted by $\Lc$. The claim of the Shiryaev-Wonham's filter is that the filtering equation becomes:
\begin{equation}
\label{eq:filtering}
\begin{split}
     d \rho^\gamma_t (x_i)
 = & \ \sum_{j} \left\{\rho^\gamma_t (x_j)  \Lc (x_j, x_i) - \rho^\gamma_t (x_i)  \Lc (x_i, x_j) \right\} dt + \sqrt{\gamma}\,  \rho^\gamma_t (x_i) \left\{ G(x_i) - \langle \rho^\gamma_t, G \rangle \right\} dW_t \ .
\end{split}
\end{equation}
Here $W$ is a $\Fc^{\yb}$-standard Brownian motion called in the filtering literature the innovation process. The quantity $\langle \rho^\gamma_t, G \rangle$ denotes the expectation of $G$ with respect to the probability measure $\rho_t^\gamma$.  Throughout the paper, we only consider $E = \{0,1\}$, i.e. the two state regime.

%

\subsubsection{Two states}
\label{subsec:2states}
 In this case, w.l.o.g $E=\{0,1\}$, and all the information is contained in 
 $$ \pi_t^\gamma := \rho^\gamma_t (1)= \P\left( \xb_t = 1 \ | \ \left({\bf y}^\gamma_s\right)_{s\le t} \right)= 1-\rho_t^\gamma (0)  \ .$$
 Making explicit in this case Eq. \eqref{eq:filtering} we observe that it has exactly the same type of dynamic as the one studied in the authors' previous paper \cite{bernardin2018spiking}. Using the notation
$$ \Lc = \begin{pmatrix}
           -\lambda_{0,1} &  \lambda_{0,1} \\
            \lambda_{1,0} & -\lambda_{1,0} \\
           \end{pmatrix} 
$$
we have indeed that Eq. \eqref{eq:filtering} can be rewritten as
\begin{align*}
d\pi^\gamma_t 
   =
   - \lambda
   \left( 
   \pi^\gamma_t - p
   \right) dt
   +
   \sqrt{\gamma} \, \sigma\,  \pi^\gamma_t \left( 1 - \pi^\gamma_t \right)
   dW_t \ ,
\end{align*}
where
\begin{equation}
\label{eq:parameters}
 \lambda = \lambda_{0,1} + \lambda_{1,0} \ ,\quad p = \lambda_{1,0}/\lambda  \ , \quad  \sigma = G^1 - G^0 \ .
\end{equation}
%
Without loss of generality, we shall assume $\sigma=1$ in the rest of the paper. Also $(G^0, G^1) = (0,1)$. In the end, our setup is indeed given by Eq. \eqref{eq:observation} and \eqref{eq:sde}, which we repeat for convenience:
\begin{align}
\label{eq:observation_repeated}
 d{\bf y}^\gamma_t = & \ \xb_t dt + \frac{1}{\sqrt{\gamma}} dB_t \ , \\
\label{eq:sde_repeated}
d\pi^\gamma_t 
   = & 
   - \lambda
   \left( 
   \pi^\gamma_t - p
   \right) \, dt
   +
   \sqrt{\gamma}\,  \pi^\gamma_t \left( 1 - \pi^\gamma_t \right)
   \, dW_t \ .
\end{align}

\begin{rmk}
\label{rmk:invariant_measure}
The invariant probability measure $\mu$ of the Markov process $\xb$ solves 
$$ \Lc^* \mu = 0
   \Longleftrightarrow
   \mu = 
   \begin{bmatrix}
   p \\ 1- p
   \end{bmatrix} \ .
$$
Without any computation, this is intuitively clear, as setting $\gamma \rightarrow 0$ yields an extremely strong observation noise and no noise in the filtering equation:
$$ d\pi^{\gamma=0}_t 
   =
   - \lambda (\pi_t^{\gamma=0} - p) dt
$$
whose asymptotic value is $p$. Informally, this says that, in the absence of information, the best estimation of the law $\Lc(\xb_t)$ in long time is the invariant measure. This is essentially the content of \cite[Theorem 4]{chigansky2006}, which holds for a Shiryaev-Wonham filter with any finite number of states.
\end{rmk}

\subsubsection{Innovation process} The innovation appearing in the SDE is the $\Fc^{\yb}$-Brownian motion $W$ obtained as:
$$ dW_t = \sqrt{\gamma} \left( d\yb_t - \langle G, \rho_t \rangle dt \right)
        = dB_t + \sqrt{\gamma} \left( G(\xb_t) - \langle G, \rho_t \rangle \right) dt \ .
$$
With the simplifying assumption that $(G^0, G^1) = (0,1)$, we obtain:
\begin{equation}
\label{eq:innovation}
dW_t = dB_t + \sqrt{\gamma} \left( \xb_t - \pi_t^\gamma \right) dt \ .
\end{equation}

\subsection{Trajectorial strong noise limits and the question} Eq. \eqref{eq:sde} falls in the scope of \cite{bernardin2018spiking} which treats the strong noise limits of a large class of one-dimensional SDEs. There the authors give a general result for SDEs not necessarily related to filtering theory. More precisely, the result is two-fold. On the one hand, the process $\pi^\gamma$ converges in a weak ``Lebesgue-type'' topology to a Markov jump process. On the other hand, if one considers a strong ``uniform-type'' topology it is possible to capture the convergence to a spike process.

\bigskip

{\bf Topology:} Fixing an arbitrary horizon time $H>0$ . The weaker topology uses the distance:
\begin{align}
\label{def:mz_topology}
{\rm d}_{\L}(f,g) := \int_0^H \left( \left| f(t) - g(t) \right| \wedge 1 \right) dt \ ,
\end{align}
inducing the Lebesgue $\L^0$ topology on the compact set $[0,H]$. This distance is the usual distance that turns the convergence in probability into a metric convergence. Notice that the previous paper \cite{bernardin2018spiking} deals with an infinite time horizon. Of course, the restricted topology there is then  the same as here.

The stronger topology is defined by using the Hausdorff distance for graphs. In this paper, a graph is nothing but a closed (hence compact) subset $\Gc=\bigsqcup_{t\in [0,H]} (\{t\} \times \Gc_t)$ of  $[0,H] \times \mathbb [0,1]$, where $\Gc_t=\{x\in [0,1]\; ; \; (t,x)\in \Gc\}$ denotes the slice of the graph $\Gc$ at time $t$. The Hausdorff distance ${\rm d}_\H (\Gc,\Gc')$  between two graphs $\Gc$ and $\Gc'$ is then defined by:

\begin{align}
\label{def:hausdorff}
   {\rm d}_{\H}(\Gc,\Gc')
 := & \ \inf\left\{ \varepsilon>0 \ |
               \ \Gc \subset \Gc' + \varepsilon \B \ ,
               \ \Gc' \subset \Gc + \varepsilon \B
               \right\} = \max_{\substack{z\in \Gc\\ z^\prime \in \Gc^\prime}} \left\{ {\rm d} (z, \Gc^\prime) , {\rm d} (z^\prime , \Gc) \right\} 
\end{align}
where $\B$ is the unit ball of ${\mathbb R}^2$ and, for $a \in \mathbb R^2$ and $B \subset \mathbb R^2$, ${\rm d} (a, B)= \inf_{b\in B} \Vert a-b \Vert$. A straightforward consequence that will be used many times in the sequel is the equivalence
\begin{align}
\label{eq:hausdorff_two_conditions}
                {\rm d}_{\H}(\Gc,\Gc') \leq \varepsilon \ 
\Longleftrightarrow & \  \left\{\begin{array}{cc}
\forall(t,g )\in\Gc , \ \exists (s,g') \in \Gc', \ \vert t-s\vert+\vert g-g'\vert\leq\varepsilon \ , \\
\forall(t,g')\in\Gc', \ \exists (s,g ) \in \Gc , \ \vert t-s\vert+\vert g-g'\vert\leq\varepsilon \ .
\end{array}
\right.
\end{align}

In particular dealing with the Hausdorff distance requires treating those two conditions. This distance is the appropriate one which allows to capture the spiking process. Indeed, when interpreting in terms of processes, this distance corresponds to the distance associated to the convergence of the graph of the processes. Spikes are then understood as vertical lines for the limit of $\pi^\gamma$. Those lines are of Lebesgue measure zero and cannot be enlightened by smoothing measure of type ${\rm d}_{\L}$. Note that the topologies usually used for the convergence of stochastic processes, such as the Skorohod topology are useless in this context. This is due to the singularity of the limiting processes as it as been pointed out in \cite{bernardin2018spiking}.

\begin{figure}[ht]
\includegraphics[scale=0.4]{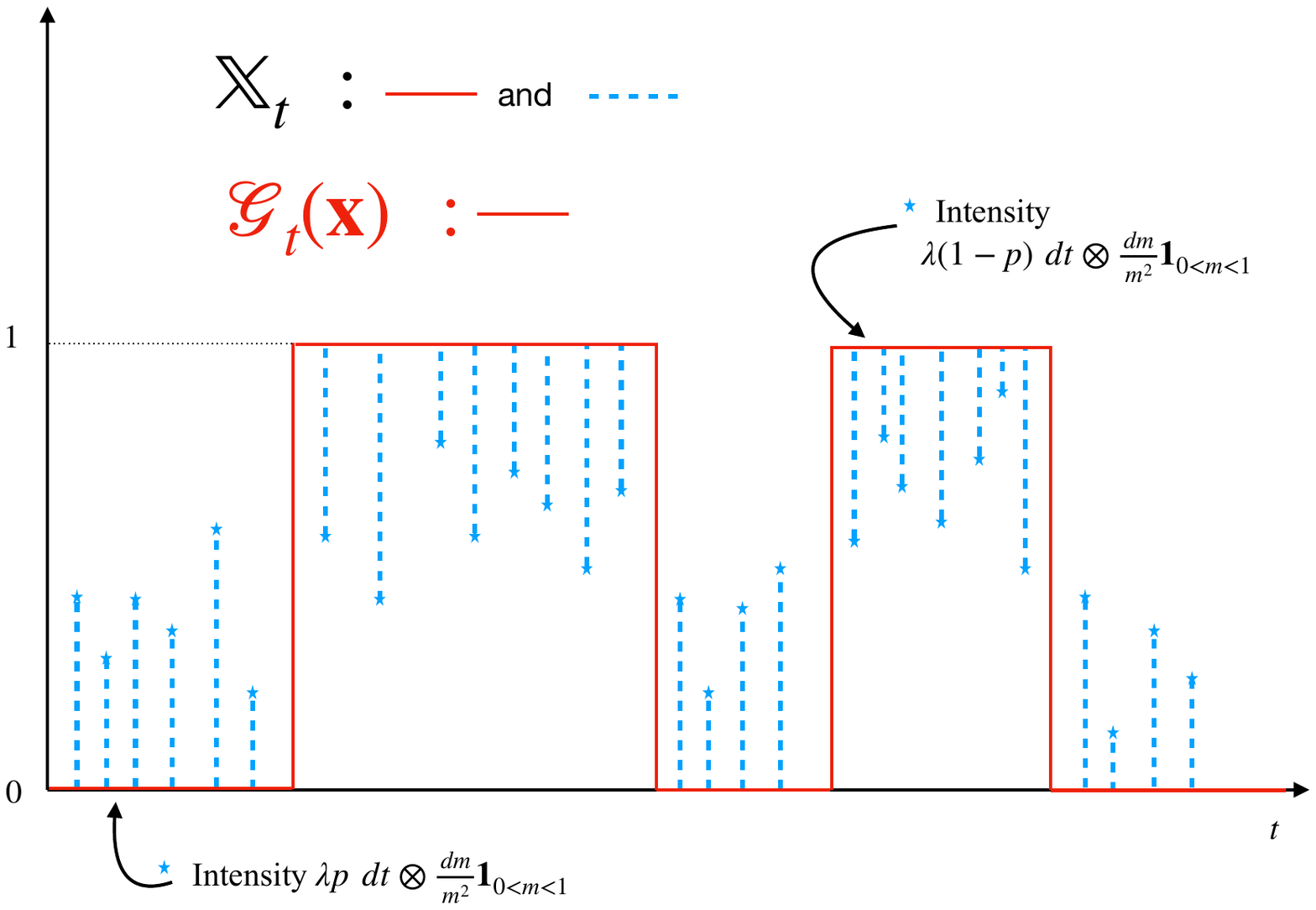}
\caption{Sketch of the two limiting processes. The graph $\mathcal G (\xb)$ of the hidden Markov pure jump process $\xb$ is in red (solid lines), and the set-valued spike process $\X$ is the union of the blue graph (dashed lines) and red graph (solid lines). }
\label{fig:sketch}
\end{figure}

\bigskip

{\bf Limiting processes:}  For the sake of completeness, we recall the construction of the spiking process which is described in \cite{bernardin2018spiking}. 

First at hand we have the process, $(\textbf x_t \; ; \; t\ge 0)$ which is a pure jump Markov process on $\{0,1\}$ with c\`adl\`ag trajectories. Recall that $\lambda p$ (resp. $\lambda (1-p)$) are the jump rate between $0$ and $1$ (resp. between $1$ and $0$), with $p\in (0,1)$ and $\lambda>0$. The initial position is sampled according to
$$ \P\left( \xb_0 = 1 \right) = 1-\P\left( \xb_0 = 0 \right) = x_0 \ .$$

Secondly, we shall define the spike process as a {\it set-valued} random path $\X: \R_+ \rightarrow \Pc\left( [0,1] \right)$,  where $\Pc\left( [0,1] \right)$ is the power set of the segment $[0,1]$. For a comprehensive sketch, see Figure \ref{fig:sketch}. It is formally obtained as follows:
\begin{itemize}
\item Sample a random initial segment $\X_0$ as
$$
   \X_0 = \left\{
\begin{array}{ll}

[Y,1] \textrm{ when $\xb_0 = 1$,} & \P\left( Y \in dy \ | \ \xb_0 = 1\right) = \frac{1-x_0}{x_0} \mathds{1}_{\{0 < y < x_0\}} \frac{dy}{(1-y)^2} \ ,\\

[0,Y] \textrm{ when $\xb_0 = 0$,} & \P\left( Y \in dy \ | \ \xb_0 = 0\right) = \frac{x_0}{1-x_0} \mathds{1}_{\{x_0 < y < 1\}} \frac{dy}{y^2} \ .
\end{array}
   \right.
$$
\item Sample $\left( t, \widetilde{M}_t \right)$ following a Poisson point process on $\R_+ \times [0,1]$ with intensity
$$ \left( dt \otimes \frac{dm}{m^2} \mathds{1}_{\{ 0 \leq m < 1 \}} \right) \ .$$
Then, by progressively rescaling time for $\left( t, \widetilde{M}_t \right)$ by
$$
   \left\{
\begin{array}{ccc}
\ \frac{1}{\lambda p} & \textrm{ when } & \xb_t = 0 \ ,\\
\ \frac{1}{\lambda(1- p)} & \textrm{ when } & \xb_t = 1 \ ,\
\end{array}
   \right. \
$$
we obtain a Poisson point process with random intensity which we denote by $\left(t, M_t \right)$. 

\item Finally
$$
   \X_t = \left\{
\begin{array}{ccc}
\ [0, M_t]   & \textrm{ if } & \xb_t = \xb_{t^-} = 0 \ ,\\
\ [1-M_t, 1] & \textrm{ if } & \xb_t = \xb_{t^-} = 1 \ ,\\
\ [0, 1]     & \textrm{ if } & \xb_t \neq \xb_{t^-} \ .
\end{array}
   \right.
$$
\end{itemize}

Notice that by virtue of $(t, M_t)$ being a Poisson point process with finite intensity away from zero, there are no points with the same abscissa and only countably many $t \in \R_+$ with $M_t>0$. If there is no point with abscissa $t \in \R_+$, then it is natural to set $M_t = 0$ and thus $\X_t = \{\xb_t\}$. This convention is natural in the sense that morally, there is always by default a point $(t, 0)$ because of the infinite measure at zero. 

\medskip

In the sequel, we call ``jump'' the set $\{t\} \times [0,1]$ when $t\ge 0$ corresponds to a jump of the process $\xb$. We also call ``spike'' a non-trivial slice $\mathbb X_t \ne \{\xb_t\}$ at a given time $t\ge 0$ which is not a jump. The ``size'' of a spike is then given by the Lebesgue measure of $\mathbb X_t$.

\bigskip

{\bf A mathematical statement:} The convergences were established thanks to a convenient (but fictitious) coupling of the processes $\left( \pi^\gamma \ ; \ \gamma > 0 \right)$ for different values of $\gamma>0$. In contrast, the filtering problem has a natural coupling for different $\gamma>0$ which is given by the observation equation \eqref{eq:observation}. In this context, let us state a small adaptation of an already established result. The precise notion of graph is given in Section \ref{subsec:FormalStatement}.

\begin{thm}[Variant of the Main Theorem of \cite{bernardin2018spiking}]
\label{thm:aap}
There is a two-faceted convergence.
\begin{enumerate}[1.]
\item {\bf In probability, for the $\L^0$ topology}, we have the following convergence in probability:
$$ \left( \pi^\gamma_ t \ ; 0 \leq t \leq H \right) 
   \stackrel{\gamma \rightarrow \infty}
            {\longrightarrow}
   \left( \xb_ t \ ; 0 \leq t \leq H \right) \ .
$$
Equivalently, that is to say
$$ \forall \varepsilon>0,
   \quad
   \lim_{\gamma \rightarrow \infty}
   \P\left( {\rm d}_\L(\pi^\gamma, \xb) > \varepsilon \right) = 0 \ .
$$
Here $\xb_0 \in \{0,1\}$ is Bernoulli distributed with parameter $\pi^\gamma_0$, the initial condition{\footnote{We assume $\pi^\gamma_0$ independent of $\gamma$.}} of $\pi^\gamma$.\\

\item {\bf In law, for the Hausdorff topology for graphs}, we have that the graph of $\mathcal G(\pi^\gamma)$ of $\left(\pi^{\gamma}_t \ ; 0\leq t\leq H \right)$ converges in law to a spike process ${\mathbb X}=\bigsqcup_{t\in [0,H]} (\{t\}\times {\mathbb X}_t)$ described by Fig. \ref{fig:sketch}.\\

\item {\bf In law, for the Hausdorff topology for graphs}, we have that the graph $\mathcal G ({\hat \xb}^\gamma)$ of $\hat \xb^\gamma = \left( \hat \xb^{\gamma}_t \ ; 0 \leq t \leq H \right)$, defined by Eq. \eqref{eq:hatx}, converges in law to another singular random closed set $\hat{\mathbb X}=\bigsqcup_{t\in [0,H]} (\{t\}\times \hat{\mathbb X}_t)$ where
$$\hat{\X}_t = \{0,1\} \mathds 1_{\{ \X_t \cap [0,\half)  \neq \emptyset,  
                                     \X_t \cap (\half, 1] \neq \emptyset
                                  \}}
             + \{0\} \mathds 1_{\{ \X_t \subset [0,\half) \}}
             + \{1\} \mathds 1_{\{ \X_t \subset (\half, 1] \}} \ .
$$
\end{enumerate}
\end{thm}

Notice that the first convergence is in the weaker Lebesgue-type topology and holds in probability, i.e. on the same probability space. The second and third convergences are in the stronger uniform-type topology, however they only hold in law, hence not necessarily on the same probability space.

\begin{proof}[Pointers to the proof]
The second point is indeed a direct corollary of \cite{bernardin2018spiking} since almost sure convergence after a coupling implies convergence in law, regardless of the coupling. Although this coupling will be used in the paper further down the road, the reader should not give it much thought for the moment.

The third point is also immediate modulo certain subtleties. Recalling that $\hat \xb_t^\gamma = \mathds{1}_{\{ \pi_t^\gamma > \half \}}$ and that the graph of $\pi^\gamma$ converges to the random closed set $\X$, it suffices to apply the Mapping Theorem \cite[Theorem 2.7]{billingsley2013convergence}. Indeed, a spike $\X_t \subset [0,1]$ is mapped to either $\{0\}$, $\{1\}$ or $\{0, 1\}$ when examining the range of the indicator $\mathds{1}_{\{\cdot > \half\}}$ on $\X_t$.  However, when invoking the Mapping Theorem, one needs to check that discontinuity points of the map $\mathds{1}_{\{\cdot > \half\}}$ have measure zero for the law of $\X$. This is indeed true since there are no spikes of height equal to $\half$ almost surely -- recall that the spike process $\X$ is described in terms of Poisson processes \cite{bernardin2018spiking}.

The first point, although simpler and intuitive, does not come from \cite{bernardin2018spiking}. In the case of filtering, the process $\xb$ is intrinsically defined, and we require the use of the specific coupling given by the additive model \eqref{eq:observation}. Let us show how the result is reduced to a single claim. The result is readily obtained from Markov inequality and the $L^1(\Omega)$ convergence:
$$ \lim_{\gamma \rightarrow \infty} 
   \E \ {\rm d}_\L\left( \pi^\gamma , \xb \right) = 0 \ .
$$
The above convergence itself only requires the definition of ${\rm d}_\L$ in Eq. \eqref{def:mz_topology}, Lebesgue dominated convergence theorem and the claim
\begin{align}
\label{eq:claim}
\forall t>0, \quad 
\lim_{\gamma \rightarrow \infty} \E \left| \pi_t^\gamma - \xb_t \right|^2 = 0 \ .
\end{align}
In order to prove Claim \eqref{eq:claim}, recall that by definition $\pi_t^\gamma$ is a conditional expectation:
\begin{align*} \pi_t^\gamma &=\ \P\left( {\bf x}_t = 1 \ | \ \left({\bf y}^\gamma_s\right)_{s\le t} \right)\\
& = \textrm{argmin}_{c \in \Fc_t^{\yb}}\  \E( \mathds{1}_{\xb_t = 1} - c )^2 \\
& = \textrm{argmin}_{c \in \Fc_t^{\yb}}\  \E( \xb_t - c )^2 
 \ .\end{align*}
At this stage, let $\varepsilon>0$ and let us introduce the process ${\mathbf z}^\varepsilon=({\mathbf z}^\varepsilon_t\ ; \ t\geq \varepsilon)$ defined for all $t\ge \varepsilon$ by
$${\mathbf z}^\varepsilon_t=\frac{1}{\varepsilon} \int_{t-\varepsilon}^t d\yb_s^\gamma \ .$$
This process is clearly $( \Fc_t^{\yb})_{t\geq \varepsilon}$ adapted, so for all $t\geq \varepsilon$, by definition of $\pi_t^\gamma$
\begin{align*}
   \E \left| \pi_t^\gamma - \xb_t \right|^2
   \leq & \
\E \left| {\mathbf z}^\varepsilon_t - \xb_t \right|^2 \\
  = & \
   \E \left| \frac{1}{\varepsilon} \int_{t-\varepsilon}^t d\yb_s^\gamma - \xb_t \right|^2 \\
   = & \
   \E \left| \frac{1}{\varepsilon} \int_{t-\varepsilon}^t \xb_s \, ds 
           - \xb_t 
           + \frac{1}{\varepsilon \sqrt{\gamma}} \int_{t-\varepsilon}^t dB_s
      \right|^2 \\
   \leq & \
   2 \E \left| \frac{1}{\varepsilon} \int_{t-\varepsilon}^t \xb_s \, ds 
           - \xb_t 
      \right|^2 
   + 
   2 \E \left| \frac{1}{\varepsilon \sqrt{\gamma}} \int_{t-\varepsilon}^t dB_s
      \right|^2 \\
   = & \
   2 \E \left| \frac{1}{\varepsilon} \int_{t-\varepsilon}^t (\xb_s - \xb_t) \, ds 
      \right|^2 
   + 
   \frac{2}{\varepsilon \gamma} \\
 \leq & \
   2 \E \left| \frac{1}{\varepsilon} \int_{t-\varepsilon}^t \mathds{1}_{\{\xb_s\neq \xb_t\}} \, ds 
      \right|^2 
   + 
   \frac{2}{\varepsilon \gamma}\\
   \leq & \ 
   2 \P\left( \xb \textrm{ jumps at least one time during } [t-\varepsilon, t] \right) 
   + 
   \frac{2}{\varepsilon \gamma} \ .
\end{align*}
Note that we have used that for $\varepsilon\leq s\leq t,$ 
$$\{\xb_s\neq \xb_t\} \subset \{\xb \textrm{ jumps at least one time during } [t-\varepsilon, t]\}.$$ Taking $\gamma \rightarrow \infty$ then $\varepsilon \rightarrow 0$ proves Claim \eqref{eq:claim}.
\end{proof}

We can now formally state the question of interest:
\begin{question}
For different regimes of $\delta = \delta_\gamma$ and $\gamma$, how do the spikes behave in the stochastic process \eqref{eq:def_process}? Basically, we need an understanding of the tradeoff between spiking and smoothing. The intuition is that there are two regimes:
\begin{itemize}
\item The slow feedback regime: the smoothing window $\delta$ is large enough so that the optimal estimator $\pi^{\delta, \gamma}$ correctly estimates the hidden process $\xb$.
\item The fast feedback regime: the smoothing window $\delta$ is too small so that $\pi^{\delta, \gamma}$ does not correctly estimate the hidden process $\xb$. One does observe the effect of spikes.
\end{itemize}
\end{question}

\subsection{Main Theorem}

Our finding is that there is sharp transition between the slow feedback regime and the fast feedback regime:
\begin{thm}[Main theorem]
\label{thm:main}
As long as $\delta_\gamma \rightarrow 0$, we have the convergence in probability, for the $\L^0$ topology, as in the first item of Theorem \ref{thm:aap}:
\begin{align}
\label{eq:cv_lesbesgue_of_smoothing}
 \left( \pi^{\delta_\gamma, \gamma}_t \ ; 0 \leq t \leq H \right) 
   \stackrel{\gamma \rightarrow \infty}
            {\longrightarrow}
   \left( \xb_ t \ ; 0 \leq t \leq H \right) \ .
\end{align}

However, in the stronger topologies, there exists a sharp transition when writing:
$$ \delta_\gamma = C \frac{\log \gamma}{\gamma} \ .$$
The following convergences hold in the Hausdorff topology on graphs in $[0,H] \times [0,1]$.

\begin{itemize}
\item (Fast feedback regime) If $C < 2$, smoothing does not occur and we have convergence in law to the spike process:
$$ \lim_{\gamma \rightarrow \infty} \pi^{\delta_\gamma, \gamma} = \X \ .$$
\item (Slow feedback regime) If $C > 8$, smoothing occurs and we have convergence:
$$ \lim_{\gamma \rightarrow \infty} \pi^{\delta_\gamma, \gamma} = \xb \ .$$
Observe that since we are dealing in this case with processes with c\`adl\`ag  paths, the convergence holds equivalently for the usual $M_2$-Skorohod topology and for the Hausdorff topology on graphs.
\end{itemize}
\end{thm}

\bigskip

\begin{proof}[Sketch of proof]
\label{sketch_of_proof}
The proof given in Theorem \ref{thm:aap} carries verbatim to proving \eqref{eq:cv_lesbesgue_of_smoothing}. We will not repeat it.

For the rest of the paper, since we only need to establish convergences in law, for the Hausdorff topology, it is more convenient to prove almost sure convergence for any coupling of the Wiener process $B = B^\gamma$ in Eq. \eqref{eq:observation}. Equivalently, we can choose a coupling of $W = W^\gamma$, which we take as the Dambis-Dubins-Schwarz coupling of \cite{bernardin2018spiking}. In that setting, we know that $\lim_{\gamma \rightarrow \infty} \pi^\gamma = \X$ almost surely, for the Hausdorff topology.
 
In Section \ref{section:smoothing}, we give in Proposition \ref{proposition:smoothing_expressions} a derivation of $\pi^{\delta_\gamma, \gamma}$ in terms of the process $\pi^{0,\gamma} =\pi^\gamma = \left( \pi^\gamma_t \ ; \ t \geq 0 \right)$. This will allow for an informal discussion explaining the phenomenon via a certain damping factor which is denoted $D_t^\gamma:=\int_{t-\delta_\gamma}^t a_s^\gamma \, ds$ in the sequel.

Before the core of the proof, we do some preparatory work in Section \ref{section:preparatory}, where we prove that only the damping term needs to be analysed. 

The core of the proof is in Section \ref{section:proof}. We start with a trajectorial decomposition of the process $\pi^\gamma$. The proof of the first statement of Theorem \ref{thm:main} is in Subsection \ref{subsection:fast_feedback}, while the proof of the second statement is in Subsection \ref{subsection:slow_feedback}.
\end{proof}

\subsection{Further remarks}
\label{subsection:further_remarks}

\bigskip

{\bf On the transition:} Without much change in the proof, one can consider $C = C_\gamma$ depending on $\gamma$. In that setting, the fast feed-back regime and the slow feed-back regime correspond respectively to 
$$ \limsup_{\gamma \rightarrow \infty} C_\gamma < 2
   \textrm{ and }
   \liminf_{\gamma \rightarrow \infty} C_\gamma > 8 \ .
$$
Furthermore, one could ask the question if there exists a threshold point $C$. See Section \ref{subsection:rmk:Cge2} for a discussion on this point. As discussed there we strongly believe that the transition is sharp i.e. the fast feed-back regime and the slow feed-back regime correspond respectively to 
$$ \limsup_{\gamma \rightarrow \infty} C_\gamma < 2
   \textrm{ and }
   \liminf_{\gamma \rightarrow \infty} C_\gamma > 2 \ .
$$
We can also ask what happens at exactly the transition $C=2$ and if there is possible zooming around the constant $C_\gamma = 2$. This matter is beyond the scope of the paper.

\bigskip

{\bf Away from the transition:} Because of the monotonicity of the damping, as a positive integral, one can easily deduce what is happening if $C_\gamma$ remains away from the threshold interval constant $[2,8]$. 

\bigskip

{\bf Is the convergence to the spike process only in law as $\gamma \rightarrow \infty$? Not in probability or almost surely?}
This point is rather subtle and we mainly choose to sweep it under the rug. Nevertheless, let us make the following comment. In the context of filtering, the spikes correspond to exceptionally fast points of the Brownian motion appearing in the noise $B=B^\gamma$. Let us assume that for some (unphysical) reason, $B^\gamma$ remains the same,  i.e. one can perfectly tune the strength of the noise at will. For different $\gamma$, the spikes appear as functionals of the Brownian motion $B$ at {\it different scales}. Therefore, we argue that there is no hope for obtaining a natural trajectorial limit to the spike process as $\gamma \rightarrow \infty$.

\bigskip

{\bf On the general Wonham-Shiryaev filter:} It is a natural question to generalise our main theorem to the Wonham-Shiryaev filter with $n$ states from Eq. \eqref{eq:filtering}. However, the mathematical technology dealing with the spiking phenomenon in a multi-dimensional setting is an open problem still under investigation.

\bigskip

{\bf Notations:} The notation $o_\gamma^\zeta (g_\gamma )$ denotes a $\gamma$ deterministic (resp. random) quantity    negligible (resp. almost surely negligible) with respect to the $\gamma$ dependent deterministic (resp. random) function $g_\gamma$,  as $\gamma$ goes to infinity, when the extra variable $\zeta$ is fixed.  Moreover,  when we are considering random quantities, and to consider convergence in probability, we denote $X_\gamma^\zeta:=o_{\gamma, \mathbb P}^\zeta\  (g_\gamma )$ a random variable such that $g_\gamma^{-1} X_\gamma^\zeta$ goes to zero  in probability as $\gamma$ goes to infinity, when the extra variable $\zeta$ is fixed. Similar usual notations are used with $o$ replaced by $\mathcal O$.  Finally  we use the notation $A \lesssim_\zeta B$ (resp. $A \gtrsim_\zeta B$) to say that there exists a finite constant $C(\zeta)>0$ (depending a priori on $\zeta$) such that $A \le C(\zeta) \ B$ (resp. $A \ge C(\zeta) B$). When the dependence in $\zeta$ is obvious, we omit the subscript $\zeta$.

When the dependence on $\zeta$ is universal (i.e. does not depend on the  parameter $\zeta$) we omit the dependence on $\zeta$ in the notation defined above.

Given a process $(\zeta_t \; ; \; t \ge 0)$ and two times $s \le t$ we denote $\zeta_{s,t} = \zeta_t - \zeta_s$ the increment of $\zeta$ between $s$ and $t$.

\section{Smoothing transform}
\label{section:smoothing}

We shall express the equation satisfied by \eqref{eq:def_process} in the $2$-states context of Section \ref{subsec:2states}. The general theory is given in \cite[Chapter 9]{LS01}. For $s\leq t$ we write:

\begin{equation}
\label{eq:pist}
\pi_{s,t}^{\gamma} := \pi_{s,t}^\gamma (1) = \ \P\left( \xb_{s} = 1 \ | \ \left({\bf y}^\gamma_s\right)_{s\le t} \right) \ .
\end{equation}

\begin{proposition}
For any $0\le s \le t$ we have that
\label{proposition:smoothing_expressions}
\begin{align}
\label{eq:single_expression}
    \pi_{s,t}^\gamma
= & \ \xb_t
    + (\pi^\gamma_{t} - \xb_t) e^{ -\int_s^t a^\gamma_u \, du }\\
  & \
    + \mathds{1}_{\{\xb_t=0\}} \int_s^t \lambda_{1,0} \frac{\pi_u^\gamma}{1-\pi^\gamma_{u}} e^{-\int_s^u a^\gamma_v \, dv} \, du 
    - \mathds{1}_{\{\xb_t=1\}} \int_s^t \lambda_{0,1} \frac{1-\pi^\gamma_{u}}{\pi_u^\gamma} e^{-\int_s^u a^\gamma_v \, dv} \, du \ ,
    \nonumber
\end{align}
where the instantaneous damping term is given by
\begin{equation}
\label{eq:a_u}
  a^\gamma_u
   :=
   a( \pi^\gamma_u )
   =
   \lambda_{1,0} \, \frac{\pi^\gamma_u}{1-\pi^\gamma_{u}}
+  \lambda_{0,1} \,  \frac{1-\pi^\gamma_u}{\pi^\gamma_{u}} \ .
\end{equation}
\end{proposition}

\begin{proof}
To simplify notation, during the proof, we forget the dependence in $\gamma$ and denote, for all $\alpha \in \{0, 1\}$,
$$ \Pi_{s,t} (\alpha) :=  \ \P\left( \xb_{s} = \alpha \ | \ \left({\bf y}^\gamma_s\right)_{s\le t} \right)\ , \quad  \Pi_{t} (\alpha) :=  \ \P\left( \xb_{t} = \alpha \ | \ \left({\bf y}^\gamma_s\right)_{s\le t} \right)\ .$$
Thanks to \cite[Theorem 9.5]{LS01}, we have:
\begin{align*}
    \partial_s \Pi_{s, t}
& = - \Pi_s \Lc\left[ \tfrac{\Pi_{s, t}}{\Pi_{s}} \right]
    + \tfrac{\Pi_{s,t}}{\Pi_{s}} 
      \Lc^{*}\left[ \Pi_{s} \right] \ , 
\end{align*}
which we will specialize to the point $\alpha=1$. Note that:
\begin{align*}
    \Pi_{s, t}(1)
& = \pi_{s,t} \ ,
\end{align*}
\begin{align*}
    \Lc\left[ \tfrac{\Pi_{s, t}}{\Pi_{s}} \right](1)
& = \lambda_{1,0} \tfrac{\Pi_{s, t}(0)}{\Pi_{s}(0)} - \lambda_{1,0} \tfrac{\Pi_{s, t}(1)}{\Pi_{s}(1)}  \ = \  \lambda_{1,0} \left( \tfrac{1-\pi_{s, t}}{1-\pi_{s}}                    - \tfrac{\pi_{s, t}}{\pi_{s}}   \right) \ ,
\end{align*}
\begin{align*}
    \Lc^*\left[ \Pi_{s} \right](1)
& = \lambda_{0,1} \Pi_{s}(0) - \lambda_{1,0} \Pi_{s}(1) \ = \  \lambda_{0,1} \left( 1-\pi_{s} \right)
   -\lambda_{1,0} \pi_{s} \ .
\end{align*}
Resuming the computation:
\begin{align*}
    \partial_s \pi_{s, t}
& = - \pi_s \lambda_{1,0} \left( \frac{1-\pi_{s, t}}{1-\pi_{s}}
                       - \frac{\pi_{s, t}}{\pi_{s}} \right)
    + \frac{\pi_{s,t}}{\pi_{s}} 
      \left( \lambda_{0,1} \left( 1-\pi_{s} \right) - \lambda_{1,0} \pi_{s} \right) \\
& = - \lambda_{1,0} \frac{\pi_s}{1-\pi_{s}}
    - \pi_s \lambda_{1,0} \left( \frac{-\pi_{s, t}}{1-\pi_{s}}
                       - \frac{\pi_{s, t}}{\pi_{s}} \right)
    + \frac{\pi_{s,t}}{\pi_{s}} 
      \left( \lambda_{0,1} \left( 1-\pi_{s} \right) - \lambda_{1,0} \pi_{s} \right) \\
& = - \lambda_{1,0} \frac{\pi_s}{1-\pi_{s}}
    + \pi_{s, t} \lambda_{1,0} \frac{1}{1-\pi_{s}}
    + \pi_{s,t}
      \left( \lambda_{0,1} \left( \frac{1}{\pi_{s}} - 1\right) - \lambda_{1,0} \right) \\
& = - \lambda_{1,0} \frac{\pi_s}{1-\pi_{s}}
    + \pi_{s, t} 
      \left[ \lambda_{1,0} \frac{\pi_s}{1-\pi_{s}}
           + 
             \lambda_{0,1} \frac{1-\pi_s}{\pi_{s}} \right] \\
& = - \lambda_{1,0} \frac{\pi_s}{1-\pi_{s}}
    + \pi_{s, t} a_s \ .
\end{align*}
One recognizes an ordinary differential equation in the variable $s$, with $s \leq t$. Upon solving, we have:
\begin{equation}
 \pi_{s,t} = \pi_{t} \ e^{ -\int_s^t a_u du }
             + \int_s^t \lambda_{1,0} \frac{\pi_u}{1-\pi_{u}} e^{-\int_s^u a_v \, dv} \, du \ .
\end{equation}
This is exactly the result when $\xb_t =0$.  

\medskip

Recall Eq. \eqref{eq:a_u}. The exact derivative:
\begin{equation*}
\label{eq:exact_derivative}
 \int_s^t \lambda_{0,1} \frac{1-\pi_{u}}{\pi_u} e^{-\int_s^u a_v dv} du
   + \int_s^t \lambda_{1,0} \frac{\pi_u}{1-\pi_{u}} e^{-\int_s^u a_v dv} du
   = \int_s^t a_u e^{-\int_s^u a_v dv} du 
   = 
   1 - e^{ -\int_s^t a_u du } 
\end{equation*}
gives the dual expression when $\xb_t =1$, and then Eq. \eqref{eq:single_expression}.

\end{proof}

\section{Reduction to the control of the damping term}
\label{section:preparatory}

\subsection{Informal discussion}

Recall that in the context of Proposition \ref{proposition:smoothing_expressions} we are interested in the case where $s=t- \delta_\gamma$ with 
\begin{equation*}
 \delta_\gamma = C\  \frac{\log \gamma}{\gamma} \ .
\end{equation*}
For $t \in [\delta_\gamma, H]$ we define thus the damping term associated to the instantaneous damping term defined by Eq. \eqref{eq:a_u}
\begin{equation}
\label{eq:dampingterm}
D_t^\gamma = \int_{t-\delta_\gamma}^t a_u^\gamma \ du =  \int_{t-\delta_\gamma}^t \left\{ \lambda_{1,0} \,  \tfrac{\pi^\gamma_u}{1-\pi^\gamma_{u}}
+  \lambda_{0,1}\,  \tfrac{1-\pi^\gamma_u}{\pi^\gamma_{u}} \right\} \ du  \ > \ 0 \ , 
\end{equation}
and the process (recall Eq. \eqref{eq:pist}) 
\begin{equation*}
\pi^{\delta_\gamma,\gamma}_t =\pi^\gamma_{t-\delta_\gamma, t} \ .
\end{equation*}

Assume there is no jumping times for $\xb$ in the time interval $[t-\delta_\gamma,t]$. Spikes, by definition, are of size strictly smaller than one. If $\pi_t^\gamma$ is collapsing on $0$, i.e. is close to zero, then $\pi_u^\gamma \approx 0$ for $u \in [t-\delta_\gamma, t]$, hence:
$$ \int_{t-\delta_\gamma}^t \lambda_{1,0}\  \tfrac{\pi^\gamma_u}{1-\pi^\gamma_{u}} \ e^{-\int_{t-\delta_\gamma}^u a^\gamma_v \, dv} \ du = o_\gamma (1)$$
and reciprocally when the collapse is on $1$. From the previous proposition, we thus have:
\begin{align*}
    \pi_{t}^{\delta_\gamma,\gamma}
= & \xb_t
    + (\pi^\gamma_{t} - \xb_t) \, e^{ - D_t^\gamma } \ + \  o_\gamma(1)
\end{align*}
Assuming that the damping term $D^\gamma$ converges to some limiting process $D$ in the large $\gamma$ limit we expect that
\begin{equation}
 \lim_{\gamma \rightarrow \infty} \pi_{t}^{\delta_\gamma, \gamma}
 = \xb_t
    + (\X_{t} - \xb_t) \, e^{ -D_t } \ .
\end{equation}
Above, the limiting graph is defined by its slice at time $t$, which is given by $\X_t$ translated by $-\xb_t$ and then rescaled by a factor $e^{ -D_t }$, and then translated by $\xb_t$ again.

\bigskip

Informally, there are three cases:
\begin{itemize}
\item Slow feedback: $D=\infty$ and therefore 
$$ \lim_{\gamma \rightarrow \infty} \pi_{t}^{\delta_\gamma, \gamma}
 = \xb_t \ .$$
\item Transitory regime: $D$ is non-trivial and therefore 
$$ \lim_{\gamma \rightarrow \infty} \pi_{t}^{\delta_\gamma, \gamma}
 = \xb_t
    + (\X_{t} - \xb_t) e^{ -D_t } \ ,$$
 with $D$ having a statistic which needs to be analyzed. This analysis is beyond the scope of this paper as mentioned in Subsection~\ref{subsection:further_remarks}.\\
\item Fast feedback: $D=0$ and therefore 
$$ \lim_{\gamma \rightarrow \infty} \pi_{t}^{\delta_\gamma, \gamma}
 = \X_{t} \ .$$
\end{itemize}

\bigskip

In this section we prove a useful intermediary step following the previous discussion, which informally says that:
\begin{align}
\label{eq:informal}
\pi_{t}^{\delta_\gamma,\gamma} \approx \xb_t + \left( \pi^\gamma_t - \xb_t \right) e^{-D_t^\gamma}. 
\end{align}

This is the combination of two simplifying facts:
\begin{itemize}
\item During jumps, Hausdorff proximity is guaranteed. Indeed, the graph of the spike process $\mathbb X$ and the graph of $\xb$ are very close in the Hausdorff sense since their slices are exactly $[0,1]$ at the jump times of $\xb$. Thus no matter where $\pi_{s,t}$ is, the Hausdorff distance will be small.
\item If $|t-s| \rightarrow 0$, away from jumps, the remainder benefits from smoothing.
\end{itemize}

Once this is established, we only need to control the damping term $D^\gamma$ outside from small spikes.

Let us now make these informal statements rigorous.

\subsection{Formal statement}
\label{subsec:FormalStatement}
Let us start with a few notations and conventions. If $f: t \in I=[a,b] \mapsto f_t \in [0,1]$, $0\le a\le b \le H$,  is a function defined on a closed subset $I$ of $[0,H]$, we denote by $\Gc (f)$ its graph:
\begin{itemize}
\item If $f$ is a continuous function continuous then we define its graph by 
\begin{equation*}
\Gc (f)=\{(t, f_t)\; ; \; t\in I\} \cup \{(t, f(a))\; ; \; 0\le t \le a\} \cup  \{(t, f(b))\; ; \; b\le t \le H\} \ . 
\end{equation*}
\item If $f$ is a  c\`adl\`ag function then, by denoting by ${\mathfrak d}_f$ the set of discontinuous points of $f$, we define its graph by
\begin{equation*}
\begin{split}
\Gc (f) &=\{(t, f_t)\; ; \; t\in I\backslash {\mathfrak d}_f \} \cup \{(t, f(a))\; ; \; 0\le t \le a\} \cup  \{(t, f(b))\; ; \; b\le t \le H\} \\
&\cup \bigcup_{ t \in {\mathfrak d}_f} (\{t\} \times [0,1]) \ .
\end{split}
\end{equation*}
\end{itemize}

Even if this definition appears to be awkward at first sight, we shall only use it for the Markov jump process $(\xb_t\; ; \; t \ge 0)$ which stands to connect the graph at jumping times
%
%
%
%

We recall that the slice at time $t$ of a graph $\mathcal G$ is denoted by $\mathcal G_t$. In order to simplify the notations, we write 
\begin{equation}
\label{eq:ggammacirct-1}
\Gc^\gamma = \Gc( \pi^{\delta_\gamma, \gamma} )
\end{equation}
for the graph induced by the process of interest (which has continuous trajectories). We also denote $\Gc^\infty := \Gc^{C,\infty}$ the candidate for the limiting graph, either the completed graphs $\X$ (if $C<2$) or $\Gc(\xb)$ (if $C>8$). By the convention above, in the definition of $\Gc(\xb)$, the graph induced by the process $\xb$, we add a vertical bar when there is a jump.
We define also $\Gc^{\gamma,\circ}$ the graph whose slice at time $t\in [\delta_\gamma, H]$ is given by 
\begin{equation}
\label{eq:ggammacirct}
\Gc^{\gamma,\circ}_t \ := \ {\mathcal G}_t (\xb) + (\pi^\gamma_{t} - \xb_t) \ e^{ -D_t^\gamma } 
\end{equation}
and the set $\mathcal G_{\delta_\gamma}^{\gamma, \circ}$ if $t\in [0, \delta_\gamma]$. Observe in particular that $\mathcal G^{\gamma,\circ}$ contains the vertical bar $[0,1]$ when there is a jump of $\xb$. 


The following formalises the informal statement of Eq. \eqref{eq:informal}:
\begin{proposition}
\label{proposition:reduction_to_damping}
Consider a coupling such that almost surely 
\begin{align}
&\lim_{\gamma \to \infty} {\rm d}_\L ( \pi^\gamma, \xb ) = 0 \ ,\\
&\lim_{\gamma \to \infty} {\rm d}_\H ({\mathcal G} (\pi^\gamma), \mathbb X) = 0 \ .
\end{align}
Then, almost surely:
\begin{align}
\label{eq:smoothing_cv_L}
&\lim_{\gamma \rightarrow \infty} {\rm d}_\L( \pi^{\delta_\gamma, \gamma}, \xb )  = 0 \ , \\
\label{eq:damping_cv}
&\lim_{\gamma \rightarrow \infty} {\rm d}_\H(\Gc^\gamma, \Gc^{\gamma,\circ} )   = 0 \ .
\end{align}
\end{proposition}

\begin{proof}
Let $J_1, J_2, \dots$ be the successive jump times of $\xb$ and let us denote $L=\sup \ \{i\ge 1\; ; \; J_i\le H\}<\infty$ the number of jumps in the time interval $[0,H]$. It is easy to prove that 
\begin{equation}
\label{eq:Omegadef}
\lim_{\eta \to 0}\mathbb P \left(  \Omega_\eta  \right) = 1 \quad \text{where the event $\Omega_\eta$ is defined by $\Omega_\eta = \left\{ \inf_{i\le L} \vert J_{i} -J_{i-1} \vert \ge 2 \eta \right\}$ } \ .
\end{equation}
We define then on $\Omega_\eta$, for $\varepsilon < \eta$, the compact sets:
\begin{equation*}
\begin{split}
&J_{\varepsilon} := \bigcup_{i=1}^L \ [J_i-\varepsilon, J_i + \varepsilon]\cap[0,H] \ ,
   \quad
   J_{\varepsilon}^\square := J_{\varepsilon} \times [0,1] \ , \\
&K_{\varepsilon} := [0,H]\backslash \mathring{J_\varepsilon} \ ,
   \quad
   K_{\varepsilon}^\square := K_{\varepsilon} \times [0,1] \ . 
\end{split}
\end{equation*}

By \cite[Theorem 1.12.15]{barnsley} we have that for any compact subsets $A,B,C,D$ of $[0,H]\times[0,1]$, 
\begin{equation}
\label{eq:hausdu}
{\rm d}_{\H} (A \cup B, C\cup D) \le \max \left\{ {\rm d}_\H (A,C), {\rm d}_\H (B,D) \right\}\le {\rm d}_\H (A,C)+{\rm d}_\H (B,D)\ .
\end{equation}
Since  $\Gc^\gamma =( \Gc^\gamma \cap J_{\varepsilon}^\square) \cup ( \Gc^\gamma \cap K_{\varepsilon}^\square)$ (and similarly for $\mathcal G^\gamma$ replaced by $\Gc^{\gamma,\circ}$) it follows that
\begin{align*}
{\rm d}_\H( \Gc^\gamma, \Gc^{\gamma,\circ} )
& \le {\rm d}_\H\left( 
   \Gc^\gamma \cap J_{\varepsilon}^\square, 
   \Gc^{\gamma,\circ} \cap J_{\varepsilon}^\square
   \right)
   +
   {\rm d}_\H\left( 
   \Gc^\gamma \cap K_{\varepsilon}^\square, 
   \Gc^{\gamma,\circ} \cap K_{\varepsilon}^\square
   \right) \ .
\end{align*}
Hence, to prove Eq. \eqref{eq:damping_cv} we only have to prove that,  a.s., on each event $\Omega_\eta$:
\begin{equation}
\label{eq:marteau2}
   \lim_{\varepsilon \rightarrow 0}
   \limsup_{\gamma \to \infty} 
   {\rm d}_\H\left( 
   \Gc^\gamma \cap K_{\varepsilon}^\square, 
   \Gc^{\gamma,\circ} \cap K_{\varepsilon}^\square
   \right)
   =
   0 \ ,
\end{equation}
and
\begin{equation}
\label{eq:marteau1}
   \lim_{\varepsilon \rightarrow 0}
   \limsup_{\gamma \to \infty} {\rm d}_\H\left( 
   \Gc^\gamma \cap J_{\varepsilon}^\square, 
   \Gc^{\gamma,\circ} \cap J_{\varepsilon}^\square
   \right)
   =
   0 \ .
\end{equation}
\bigskip

We divide the proof of the proposition in three steps: we first prove Eq. \eqref{eq:marteau2}, then Eq. \eqref{eq:smoothing_cv_L} and then Eq. \eqref{eq:marteau1}.

\bigskip
\noindent
{\bf Step 1: Hausdorff proximity away from the jump times: proof of Eq. \eqref{eq:marteau2}.}
\bigskip

{\underline{Step 1.1: Spikes are of size less than $1-\eta$ with high probability.}}	
\medskip

Let $M^*$ be the largest length of a spike:
\begin{equation}
\label{eq:biggestspike}
M^* := \max_{t \in [0,H]} \max_{y \in \X_t \backslash {\mathcal G}_t (\xb) } \left\{ \left| y \right|\mathds{1}_{\xb_t=0} , (1-\left| y\right|)\mathds{1}_{\xb_t=1} \right\} \ .
\end{equation}
From the explicit description of the law of $\X$, $M^*$ is the maximum decoration of a Poisson process on $[0,H] \times [0,1]$ with intensity 
$$
  \left(
  p \mathds{1}_{\{ \xb_t=0 \}}
  +
  (1-p) \mathds{1}_{\{ \xb_t=1 \}}
  \right)
  \lambda dt \otimes \frac{dm}{m^2}
  \ .
$$
Upon conditioning on the process $\xb$, and considering the definition of a Poisson process \cite[\textsection 2.1]{kingman1992poisson}, notice that the number of points falling inside $[0,H]\times (1-\eta,1)$ is a Poisson random variable with parameter
$$ 
   \int_{[0,H]\times (1-\eta,1)} 
                         \left(
  p \mathds{1}_{\{ \xb_t=0 \}}
  +
  (1-p) \mathds{1}_{\{ \xb_t=1 \}}
  \right)
  \lambda
  dt \otimes \frac{dm}{m^2} \ .
$$
As such the event $\{ M^*\le 1-\eta \}$ corresponds to having this Poisson random variable being zero, so that:
\begin{equation}
\label{eq:estimspik}
\begin{split}
    & \P\left( M^*\le 1-\eta \right) \\
  = & \ \E \exp\left( -\int_{[0,H]\times (1-\eta,1)} 
                         \left(
  p \mathds{1}_{\{ \xb_t=0 \}}
  +
  (1-p) \mathds{1}_{\{ \xb_t=1 \}}
  \right)
  \lambda
  dt \otimes \frac{dm}{m^2} \right)\\
  = & \ \E \exp\left( - \lambda \left( 1/(1-\eta) - 1 \right)
                         \int_0^H
                         \left(
  p \mathds{1}_{\{ \xb_t=0 \}}
  +
  (1-p) \mathds{1}_{\{ \xb_t=1 \}}
  \right)
  dt  \right)\\
   = &\  \E \exp\left( -\frac{\lambda \eta}{1-\eta}
	\int_0^H
  \left(
  p \mathds{1}_{\{ \xb_t=0 \}}
  +
  (1-p) \mathds{1}_{\{ \xb_t=1 \}}
  \right)
  dt   
   \right)\\
   \geq & \ 
   \exp\left( -\frac{\lambda \eta}{1-\eta} H \max(p, 1-p) 
   \right) \ .
   \end{split}
\end{equation}
As such, it is clear from the last inequality of Eq. \eqref{eq:estimspik} that 
\begin{equation}
\label{eq:Mstar}
\lim_{\eta \to 0} \mathbb P (M^* \leq 1-\eta) = 1 \ .
\end{equation}

\medskip

{\underline{Step 1.2: End of the proof of Eq. \eqref{eq:marteau2}.}} 
\medskip

We observe now, by definition of Hausdorff distance, that for any $t\in K_\varepsilon$ and $u \in [t-\delta_\gamma, t]$ there exists $s\in [0,H]$ and $x\in \mathbb X_s$ such that
\begin{equation*}
\| (s,x) - (u, \pi_u^\gamma) \|^2
=
\vert s-u\vert^2 + \vert x- \pi_u^\gamma \vert^2  
\leq 
{\rm d}^2_{\H} \left( \Gc (\pi^\gamma), \mathbb X \right) \ .
\end{equation*}
From the definition of $M^*$, it implies that 
\begin{equation}
\label{eq:inf}
\sup_{\substack{t \in K_\varepsilon\\ {\xb}_t=0}} \; \sup_{u\in [t-\delta_\gamma, t]} \pi_u^\gamma \le  M^* + {\rm d}_{\H} \left( \Gc (\pi^\gamma), \mathbb X \right) 
\end{equation}
and
\begin{equation}
\label{eq:sup}
\begin{split}
\inf_{\substack{t \in K_\varepsilon\\ {\xb}_t=1}} \; \inf_{u\in [t-\delta_\gamma, t]} \pi_u^\gamma  \ge  (1-M^*) - {\rm d}_{\H} \left( \Gc (\pi^\gamma), \mathbb X \right) \ .
\end{split}
\end{equation}
Recalling Eq. \eqref{eq:Omegadef} and Eq. \eqref{eq:Mstar}, let us then denote the event 
\begin{equation}
\label{eq:Omegaprime}
\Omega_{\eta}' := 
\Omega_\eta \cap 
\{ M^* \leq 1 - 2 \eta \} \ ,
\end{equation}
which satisfies $\liminf_{\eta \to 0} \mathbb P (\Omega_{\eta}')=1$, and on which we have
\begin{align}
\label{eq:inf-sup}
\min\left( M_0^\gamma, M_1^\gamma \right) \geq &  \ 2\eta - {\rm d}_{\H} \left( \Gc (\pi^\gamma), \mathbb X \right) 
\end{align} 
where
\begin{align}
M_0^\gamma := 
\inf_{\substack{t \in K_\varepsilon\\ {\xb}_t=0}} \; \inf_{u\in [t-\delta_\gamma, t]} \{1-\pi_u^\gamma\}
\quad  & \text{and} \quad 
M_1^\gamma :=
\inf_{\substack{t \in K_\varepsilon\\ {\xb}_t=1}} \; \inf_{u\in [t-\delta_\gamma, t]} \pi_u^\gamma \ .
\nonumber
\end{align}
Notice in particular that:
\begin{equation}
\label{eq:inf-sup2}
\limsup_{\gamma \rightarrow \infty} 
\left\{ \frac{1}{M_0^\gamma}
+
\frac{1}{M_1^\gamma} \right\}
\le \frac{1}{\eta} \ .
\end{equation}
Recall Eq. \eqref{eq:single_expression}. For $t \in [\delta_\gamma,H]\backslash J_\varepsilon$, on the event ${\Omega}_{\eta}'$ (defined by Eq. \eqref{eq:Omegaprime}), we have thanks to Eq. \eqref{eq:inf-sup}, that, for $\gamma$ sufficiently large,
\begin{align*}
  &   
  \left|  \pi_{t}^{\delta_\gamma, \gamma}
       - \left( \xb_t
    + (\pi^\gamma_{t} - \xb_t) \,  e^{ -D_t^\gamma } \right) 
  \right|\\
\leq &\  \mathds{1}_{\{\xb_t=0\}} \int_{t-\delta_\gamma}^t \lambda_{1,0} \frac{\pi_u^\gamma}{1-\pi^\gamma_{u}} \, du 
     + \mathds{1}_{\{\xb_t=1\}} \int_{t-\delta_\gamma}^t \lambda_{0,1} \frac{1-\pi^\gamma_{u}}{\pi_u^\gamma} \, du \\
\leq & \ \left( \frac{1}{M_0^\gamma} + \frac{1}{M_1^\gamma} \right) 
     \left(
     \mathds{1}_{\{\xb_t=0\}} \int_{t-\delta_\gamma}^t \pi_u^\gamma \, du 
     + \mathds{1}_{\{\xb_t=1\}} \int_{t-\delta_\gamma}^t (1-\pi^\gamma_{u}) \, du 
     \right)\\
\stackrel{(*)}{=}
     & \ \left( \frac{1}{M_0^\gamma} + \frac{1}{M_1^\gamma} \right) 
     \left(
     \mathds{1}_{\{\xb_t=0\}} \int_{t-\delta_\gamma}^t (\pi_u^\gamma - \xb_u) \, du 
     + \mathds{1}_{\{\xb_t=1\}} \int_{t-\delta_\gamma}^t (\xb_u-\pi^\gamma_{u}) \,  du 
     \right) \\
=    & \ \left( \frac{1}{M_0^\gamma} + \frac{1}{M_1^\gamma} \right) 
     \int_{t-\delta_\gamma}^t \left| \pi_u^\gamma - \xb_u \right| \, du \\
\leq & \ \left( \frac{1}{M_0^\gamma} + \frac{1}{M_1^\gamma} \right) \delta_\gamma \ .
\end{align*}
The step marked with (*) holds because there is no jump during $[t-\delta_\gamma, t]$ for $t \in K_\varepsilon$ as soon as $\delta_\gamma < \varepsilon$. Taking limits and using Eq. \eqref{eq:inf-sup2}, we conclude that:
\begin{align}
  \label{eq:K_epsilon_control}
  \limsup_{\gamma \rightarrow \infty}
  \sup_{t \in K_\varepsilon}
  \left|  \pi_{t}^{\delta_\gamma, \gamma}
       - \left( \xb_t
    + (\pi^\gamma_{t} - \xb_t) \, e^{ -D_t^\gamma } \right) 
  \right|
\leq & \ \limsup_{\gamma \rightarrow \infty} \frac{\delta_\gamma}{\eta} = 0 \ .
\end{align}
This limit holds on $\Omega_\eta'$ for all $\eta>0$. We have thus proven that away from jumps the Hausdorff distance tends to zero. This concludes the proof of Eq. \eqref{eq:marteau2}.

\bigskip
\noindent
{\bf Step 2: Proof of Eq. \eqref{eq:smoothing_cv_L}: $\lim_{\gamma \rightarrow \infty} {\rm d}_\L( \pi^{\delta_\gamma, \gamma}, \xb )  = 0$.}
\bigskip

From the definition of the distance ${\rm d}_\L$
\begin{align*}
 {\rm d}_\L( \pi^{\delta_\gamma, \gamma}, \xb )
= & \ \int_0^H  \left| \pi^{\delta_\gamma, \gamma}_t  - \xb_t \right| dt =\int_{J_\varepsilon}  \left| \pi^{\delta_\gamma, \gamma}_t  - \xb_t \right| dt+\int_{K_\varepsilon} \left| \pi^{\delta_\gamma, \gamma}_t  - \xb_t \right| dt \\
\leq & \ 2 \varepsilon L + \int_{K_\varepsilon} \left| \pi^{\delta_\gamma, \gamma}_t  - \xb_t \right| dt \\
\leq & \ 2 \varepsilon L + {\rm d}_\L\left( \pi^\gamma, \xb \right) + 
       \int_{K_\varepsilon} \left| \pi^{\delta_\gamma, \gamma}_t  - 
       \left( \xb_t + ( \pi^\gamma - \xb_t ) e^{-D^\gamma_t} \right) \right| dt 
       \ \ .
\end{align*}
Now, notice that because $\lim_{\gamma \rightarrow \infty} {\rm d}_\L\left( \pi^\gamma, \xb \right) = 0$ and the estimate from Eq. \eqref{eq:K_epsilon_control}, we have:
$$ \limsup_{\gamma \rightarrow \infty} {\rm d}_\L( \pi^{\delta_\gamma, \gamma}, \xb )
\leq 2 \varepsilon L $$
for all $\varepsilon>0$. Since $L<\infty$ almost surely, this concludes the proof of Eq. \eqref{eq:smoothing_cv_L}.

\bigskip
\noindent
{\bf Step 3: Hausdorff proximity around jump times: proof of Eq. \eqref{eq:marteau1}.}
\medskip

Thanks to the triangle inequality, to prove Eq. \eqref{eq:marteau1}, it is sufficient to prove that 

\begin{equation}
\label{eq:perceuse}
   \limsup_{\gamma \to \infty} 
   {\rm d}_\H\left( 
   \Gc^\gamma \cap J_{\varepsilon}^\square,  J_{\varepsilon}^\square
   \right)
   \lesssim \varepsilon 
\end{equation}
and
\begin{equation}
\label{eq:perceuse3}
   \limsup_{\gamma \to \infty}
   {\rm d}_\H\left( 
   \Gc^{\gamma,\circ}  \cap J_{\varepsilon}^\square,  J_{\varepsilon}^\square
   \right)
   \lesssim \varepsilon \ .
\end{equation}

For Eq. \eqref{eq:perceuse}, it is then sufficient to show that on $\Omega_\eta$:
\begin{align}
\label{eq:101}
   \forall t \in J_\varepsilon,\quad  \exists (s,y) \in J_\varepsilon^\square, & 
   \quad \left| (t,\pi_t^{\delta_\gamma, \gamma}) - (s,y) \right|
   \lesssim \varepsilon   \ ,
   \\
   \label{eq:102}
   \forall (t,x) \in J_\varepsilon^\square, \quad
   \exists s \in J_\varepsilon, & 
   \quad \left| (s,\pi_s^{\delta_\gamma, \gamma}) - (t,x) \right| 
   \lesssim \varepsilon \ .
\end{align}

The first inequality \eqref{eq:101} is readily obtained by noticing that $J_\varepsilon^\square$ contains vertical lines at the moment of jumps:
$$ \cup_i \{J_i\} \times [0,1] \subset J_\varepsilon^\square \ .$$
As such, we simply need to pick $s=J_i$ and $y= \pi_t^{\delta_\gamma, \gamma}$, where $i$ is such that $t\in [J_i-\varepsilon, J_i + \varepsilon]$. 

For the second inequality \eqref{eq:102} we notice that we can assume that $t$ is a jump time of $\xb$ since any $(t,x) \in J_\varepsilon^\square$ is at distance at most $\varepsilon$ from a jump time. Hence we assume $t=J_i$ for some $i$. Observe now that
\begin{align*}
    & \inf_j \inf_{s \in [J_j-\varepsilon, J_j+\varepsilon]} \vert {\xb}_s- \pi_{s}^{\delta_\gamma, \gamma} \vert \\
\le \ & \frac{1}{2\varepsilon} \sum_j  \int_{s \in [J_j-\varepsilon, J_j+\varepsilon]} \vert {\xb}_s- \pi_{s}^{\delta_\gamma, \gamma} \vert \, ds \\
\le \ & \frac{1}{2\varepsilon} \int_{0}^H \vert {\xb}_s- \pi_{s}^{\delta_\gamma, \gamma} \vert \, ds
\\=\ & \frac{1}{2\varepsilon} {\rm d}_\L( \pi^{\delta_\gamma, \gamma}, \xb )
\end{align*}
which goes to $0$ as $\gamma$ goes to infinity by Eq. \eqref{eq:smoothing_cv_L}.  Observe that the process $\xb$ takes different values in $[J_i -\varepsilon, J_i)$ and in $(J_i, J_i+\varepsilon]$. The previous bound implies that
\begin{equation}
\label{eq:IVT}
\forall \ 1\le j \le L, \quad \exists \,  s, s' \in [J_j-\varepsilon, J_j+ \varepsilon] \quad \text{\rm{ s.t. }} \quad \pi_{s}^{\delta_\gamma, \gamma} \le \eta_\gamma, \quad 1-\pi_{s'}^{\delta_\gamma, \gamma} \le \eta_\gamma
\end{equation}
where $\lim_{\gamma \to \infty} \eta_\gamma =0$. Because $\pi^{\delta_\gamma, \gamma}$ is continuous, the Intermediate Value Theorem states that Eq. \eqref{eq:102} is satisfied for $\gamma$ large enough. Hence we have proved Eq. \eqref{eq:perceuse}.

To prove Eq. \eqref{eq:perceuse3}, we recall that $\mathcal G^{\gamma,\circ}$ (defined by Eq. \eqref{eq:ggammacirct}) contains the vertical bar $[0,1]$ when there is a jump of $\xb$. It follows then immediately that
\begin{align}
   \forall (t,x) \in \Gc^{\gamma,\circ} \cap J_\varepsilon^\square,\quad  \exists (s,y) \in J_\varepsilon^\square, & 
   \quad \left| (t,x) - (s,y) \right|
   \lesssim \varepsilon   \ ,
   \\
   \forall (t,x) \in J_\varepsilon^\square,\quad  \exists (s,y) \in\Gc^{\gamma,\circ} \cap J_\varepsilon^\square, & 
   \quad \left| (t,x) - (s,y) \right|
   \lesssim \varepsilon   \ ,
\end{align}
and Eq. \eqref{eq:perceuse3} then follows.
\end{proof}

\section{Proof of Main Theorem: Study of the damping term}
\label{section:proof}

Thanks to Proposition \ref{proposition:reduction_to_damping} the establishment of Theorem  \ref{thm:main} is reduced to the proof of 
\begin{equation*}
\lim_{\gamma \to \infty} {\rm d}_{\mathbb H} \left({\Gc}^{\gamma, \circ}, \Gc^\infty  \right) =0 \ .
\end{equation*}
Above, as mentioned above, the convergence is understood almost surely since we can always assume the existence of a coupling between the processes involved in this limit. We recall that $\Gc^{\gamma, \circ}$ is defined via Eq. \eqref{eq:ggammacirct}. In Proposition \ref{proposition:more_reduction_to_damping} we will show that this can be done by showing the two following facts:
\bigskip
\begin{itemize}
\item If $C<2$ then
\begin{equation}
\label{eq:mt-01}
\lim_{\gamma \to \infty} D_t^\gamma =0\ ,
\end{equation}
for the relevant times $t$, which correspond to a spike.\\
\item If $C>8$ then
\begin{equation}
\label{eq:mt-02}
\lim_{\gamma \to \infty} D_t^\gamma =\infty\ .
\end{equation}
again, for the relevant times $t$ corresponding to a spike.
\end{itemize}

\medskip 

In order to prove Proposition \ref{proposition:more_reduction_to_damping} we need to introduce several definitions. 

\subsection{Decomposition of trajectory}
%
%

Recall that without loss of generality, we may assume the almost sure convergence of $\pi^\gamma$ to the spike process $\X$ -- see the discussion in the sketch of proof of the Main Theorem \ref{thm:main}.

Let $\varepsilon>0$ be a sufficiently small positive number, which will go to zero at the end of the proof, but after the large $\gamma$ limit. We define a sequence of stopping times by $S^\gamma_0 (\varepsilon) = 0$ and, via induction on $j \geq 1$, by (see Fig. \ref{fig:traj}), 
$$ T^\gamma_{j-1} (\varepsilon)
   :=
   \inf\left\{ t \geq S_{j-1}(\varepsilon) 
             \ \Big| \ \pi^\gamma_t \leq \tfrac{\varepsilon}{2} \quad \text{or} \quad \pi^\gamma_t \geq 1-\tfrac{\varepsilon}{2} \right\} \ ,
$$

$$ S^\gamma_j (\varepsilon)
   :=
   \inf\left\{ t \geq T_{j-1}^\gamma (\varepsilon)  
             \ \Big| \ \pi^\gamma_t \geq \varepsilon  \quad \text{and} \quad \pi^\gamma_t \leq 1-\varepsilon\right\} \ .
$$
To enlighten the notation, the dependence in $\gamma$ and $\varepsilon$ of these stopping times is in the sequel usually omitted. The $S_j$'s have to be understood as the \textit{starting} times of spikes (or jumps), and the $T_j$'s have to be understood as the \textit{terminating} times of spikes (or jumps). Observe there exists a finite random variable $N:=N_{\varepsilon}^\gamma$ such that $S_N >H$, i.e. almost surely  there are $N$ intervals $[S_j, T_j]$ completely included in $[0,H]$. Indeed, we know from our previous work \cite{bernardin2018spiking}, that $\pi^\gamma$ converges a.s. to $\X$ for the Hausdorff topology on graphs as $\gamma$ goes to infinity. Also, for any $\varepsilon >0$, there are finitely many spikes of size larger than $\varepsilon$ (for $\mathbb X$). Therefore $N^\gamma_\varepsilon$ is necessary a.s. bounded independently of $\gamma$. 


\begin{center}
\begin{figure}
  \includegraphics[width=0.8\linewidth]{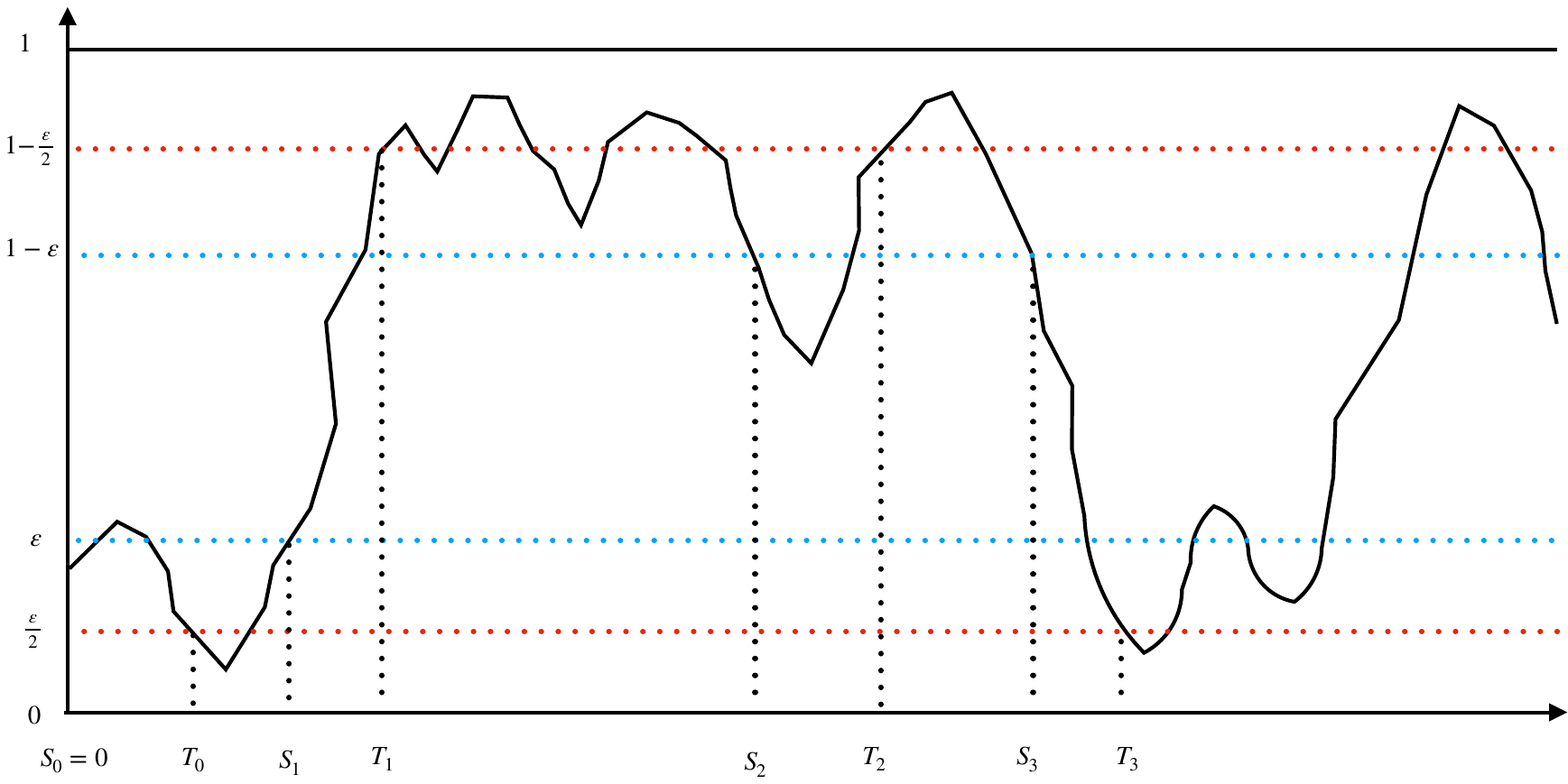}
  \caption{Decomposition of trajectory.}
  \label{fig:traj}
\end{figure}
\end{center}

\medskip

Let us start with a useful lemma, which will permit, in the proof of Proposition \ref{proposition:more_reduction_to_damping}, to avoid to control the damping outside the excursion intervals $\bigsqcup_j [S_j, T_j]$.

\begin{lemma}
\label{lemma:reduction}
Assume $0 \leq {\rm d}_\H( \pi^\gamma, \X ) \leq \varepsilon  < \eta < \half$. 
For all $t \in [0,H] \backslash \bigsqcup_j [S_j, T_j]$, we have:
$$ \left| \pi_t^\gamma - \xb_t \right| \leq \varepsilon \ ,$$
on the event $\Omega^\prime_\eta=\{ M^* < 1 - 2 \eta \} \cap \Omega_\eta$ defined by Eq. \eqref{eq:Omegaprime} and $M^*$ by Eq. \eqref{eq:biggestspike}.
\end{lemma}
\begin{proof}
By definition, we have for such times $t$:
$$ \pi_t^\gamma \wedge (1-\pi_t^\gamma) \leq \varepsilon \ ,$$
so that the natural estimator for $\bf x_t$ is
$ \hat{\xb}_t := \mathds{1}_{\{ \pi_t^\gamma > \half \}}$ .
By definition of Hausdorff distance, there exists a pair $(s,x) \in [0,H] \times \X_s$ such that
$$ |t-s|^2 + |x-\pi_t^\gamma|^2 \leq {\rm d}_\H( \pi^\gamma, \X )^2 \ ,$$
which implies
$$ |x - \hat{\xb}_t | 
\leq |x - \pi_t^\gamma | + |\pi_t^\gamma - \hat{\xb}_t | 
\leq {\rm d}_\H( \pi^\gamma, \X ) + \varepsilon
\leq 2 \varepsilon \ .
$$
On the event $\{ M^* < 1 - 2 \eta \}$, it entails that 
$$ | x - \xb_s | \leq M^* < 1 - 2 \eta \ .$$
Necessarily 
$$ | \xb_s  - \hat{\xb}_t | < 1 - 2\eta + 2 \varepsilon < 1  \ ,$$
which amounts to equality. Using that $|s-t| \leq \varepsilon < \eta$, there are no jumps between $s$ and $t$ on the event $\Omega_\eta$. We thus have $\xb_s = \xb_t = \hat{\xb}_t$.
\end{proof}

We consider the event $\Omega^\prime_\eta=\{ M^* < 1 - 2 \eta \} \cap \Omega_\eta$ defined by Eq. \eqref{eq:Omegaprime} and on such event, for or any $\varepsilon < \eta$, we denote the finite random set
\begin{equation}
\label{eq:Iepsilongamma}
I_{\varepsilon}^{\gamma} := 
\left\{ j \in \{ 1, \ldots, N^\varepsilon_\gamma\} \; ; \; [S_j^\gamma (\varepsilon), T_j^\gamma (\varepsilon)] \cap K_\varepsilon \neq \emptyset \right\} \ .
\end{equation}

\bigskip
{\bf Separation argument:} See Figure \ref{fig:sep} for a comprehensive graphical explanation of the following argument. A single segment $[S_j, T_j]$ corresponds, in the large $\gamma$ limit, to either a spike of size larger than $\varepsilon$, or a jump. Because $[S_j, T_j] \cap K_\varepsilon \neq \emptyset$, we are far from jumps and the segment $[S_j, T_j]$ necessarily corresponds to a spike. Notice that multiple $[S_j, T_j]$ can correspond to the same spike $\{s\}\times \X_s$ in the limit. 

Therefore for a spike $\{s\} \times \mathbb X_s$ of size larger than $\varepsilon$, with $0\le s\le H$, we denote by $I_\varepsilon (s)$ the random finite set of indexes $j$'s such that the interval $[S_j, T_j]$ asymptotically coalesces to the time location $s$ of the spike:
\begin{equation*}
\begin{split}
I_\varepsilon (s) 
  & = \left\{ j \in {\mathbb N} \; ; \; \lim_{\gamma \to \infty} S_j^\gamma (\varepsilon) = \lim_{\gamma \to \infty} T_j^\gamma (\varepsilon) =  s \right\} \ . 
   \end{split}
\end{equation*}
The equality between the two limits follows from Corollary 2.4 in \cite{bernardin2018spiking}. Indeed, by this corollary we know that the time spent by $\pi^\gamma$ in the interval $[\varepsilon/2, 1-\varepsilon/2]$ during the time window $[0,H]$ is of order $\Oc_\gamma^\varepsilon(1/\gamma)$. Hence $ \sum_j |T_j-S_j| = \Oc_\gamma^\varepsilon(1/\gamma)$. Since $(\pi^\gamma)_\gamma>0$ converges a.s. to $\mathbb X$ in the Hausdorff topology in the large $\gamma$ limit, this implies the existence of the limits above.


\begin{figure}
\centering
\includegraphics[width=0.8\linewidth]{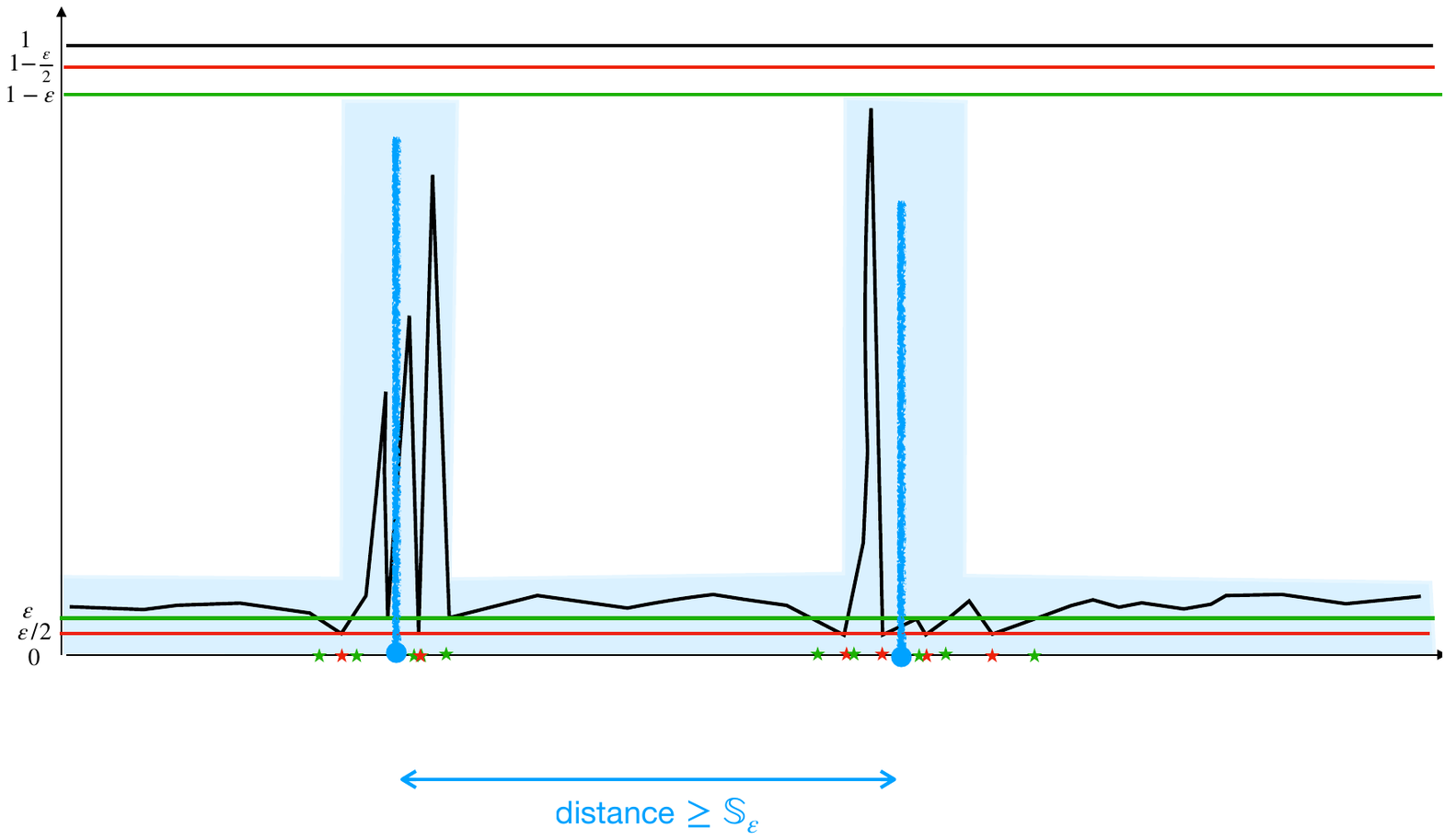}
\caption{{\bf{The separation argument}}. In blue is represented the spike process restricted on some interval of time where, to have a comprehensive picture, we have only two spikes with time position $s$ and $s'$ of length bigger than $\varepsilon$ and separated by a time-distance at least $\mathbb S_{\varepsilon}$. The spikes of size smaller than $\varepsilon$ are not represented. The green stars correspond to the stopping times $S_j$'s and the red stars to the stopping times $T_j$'s. The black curve representing the trajectory of $\pi^\gamma$ is contained in a ${\rm d}_\H (\Gc (\pi^\gamma), \mathbb X)$-thickening (painted in blue) of the blue graph of $\mathbb X$.}
  \label{fig:sep}
\end{figure}

The spikes (for $\mathbb X$) larger than $\varepsilon$ are separated by a random constant ${\mathbb S}_\varepsilon>0$ and,  by definition of Hausdorff distance, we have then that{\footnote{The Hausdorff distance in dimension one is defined similarly to the one in dimension two.}}:
\begin{align}
\label{eq:S_epsilon_def}
 s \neq s' \implies 
 & {\rm d}_{\mathbb H} \left( \cup_{j \in I_\varepsilon (s) } [S_{j}, T_{j}], \cup_{j \in I_{\varepsilon}  (s')} [S_{j}, T_{j}] \right) > {\mathbb S}_\varepsilon - {\rm d}_\H (\Gc (\pi^\gamma), \mathbb X)  \ .
\end{align}

Thanks to this, supposing the spike at $s$ is starting from $0$, i.e. $\pi_{S_j}^\gamma =\varepsilon$, we can strengthen the claim:
$$ \forall u \in [S_j, T_j], \quad 1 - \tfrac{\varepsilon}{2} \geq \pi^\gamma_u $$
for any $j \in I_\varepsilon (s)$ to\\
\begin{equation}
\label{eq:strenghtened_claim}
\begin{split}
&{\text {For any}} \quad  u \in \left[S_j \wedge  \left( T_j - {\mathbb S}_\varepsilon + {\rm d}_\H (\Gc (\pi^\gamma), \mathbb X)  \right), \ T_j \vee \left( S_j + {\mathbb S}_\varepsilon - {\rm d}_\H (\Gc (\pi^\gamma), \mathbb X)  \right)\right]  \ ,\\
&{\text {we have that }}  \quad 1 - \tfrac{\varepsilon}{2} \geq \pi^\gamma_u \ .
\end{split}
\end{equation}

Similarly, supposing the spike at $s$ is starting from $1$, i.e. $\pi_{S_j}^\gamma =1-\varepsilon$, we can strengthen the claim:
$$ \forall u \in [S_j, T_j], \quad \tfrac{\varepsilon}{2} \leq \pi^\gamma_u $$
for any $j \in I_\varepsilon(s)$ to\\
\begin{equation}
\begin{split}
\label{eq:strenghtened_claim2}
&{\text {For any}} \quad  u \in \left[S_j \wedge  \left( T_j - {\mathbb S}_\varepsilon + {\rm d}_\H (\Gc (\pi^\gamma), \mathbb X)  \right), \ T_j \vee \left( S_j + {\mathbb S}_\varepsilon - {\rm d}_\H (\Gc (\pi^\gamma), \mathbb X)  \right)\right]  \ ,\\
&{\text {we have that }} \quad  \tfrac{\varepsilon}{2} \leq \pi^\gamma_u \ .
\end{split}
\end{equation}

\medskip

\begin{proposition}
\label{proposition:more_reduction_to_damping}
Recall the definition of the damping term $D^\gamma$ given in Eq. \eqref{eq:dampingterm} and of the set $I_\varepsilon^\gamma$ in Eq. \eqref{eq:Iepsilongamma}. Assume that either\\
\begin{align}
\label{eq:no_damping}
C < 2 \quad \textrm{ and } \quad 
\forall \varepsilon>0, \quad 
\lim_{\gamma \rightarrow \infty}\
\sup_{j\in I_\varepsilon^\gamma} \ \sup_{t \in [S_j^\gamma (\varepsilon) , T_j^\gamma (\varepsilon) ]}
D_t^\gamma =  0 \ 
\end{align}
or 
\begin{align}
\label{eq:damping}
C > 2 \quad \textrm{ and } \quad 
\forall \varepsilon>0, \quad
\lim_{\gamma \rightarrow \infty} \ 
\inf_{j\in I_\varepsilon^\gamma} \ \inf_{t \in [S^\gamma_j (\varepsilon ), T_j^\gamma (\varepsilon)]}
D_t^\gamma
= \infty \ .
\end{align}
Then we have
\begin{equation*}
\lim_{\gamma \to \infty} {\rm d}_\H ( \mathcal G^\gamma, \mathcal G^\infty)  =0 \ .
\end{equation*}

\end{proposition}

\begin{proof}
By Proposition \ref{proposition:reduction_to_damping} and the triangle inequality it is sufficient to prove 
\begin{equation*}
\lim_{\gamma \to \infty} {\rm d}_\H ( \mathcal G^{\gamma,\circ}, \mathcal G^\infty)  =0 \ .
\end{equation*}

\bigskip
\noindent
{\bf Case 1: $C<2$.} Then $\mathcal G^\infty = \mathbb X$ and thanks to the triangle inequality:
$$ {\rm d}_\H ( \mathcal G^{\gamma,\circ}, \mathbb X)
   \leq
   {\rm d}_\H ( \mathcal G^{\gamma,\circ}, \mathcal G\left( \pi^\gamma \right) )
   +
   {\rm d}_\H ( \mathcal G\left( \pi^\gamma \right), \X) \ .
$$
The second term goes to zero thanks to Theorem \ref{thm:aap} of \cite{bernardin2018spiking}. It thus suffices to show that
\begin{equation}
\label{eq:onemore}
\lim_{\gamma \to \infty} {\rm d}_\H ( \mathcal G^{\gamma,\circ}, \mathcal G\left( \pi^\gamma \right) ) = 0 \ .
\end{equation}
As such, starting with the definition of Hausdorff distance, we have that:
\begin{equation}
\label{eq:bobound22}
\begin{split}
  {\rm d}_\H ( \mathcal G^{\gamma,\circ}, \mathcal G\left( \pi^\gamma \right) )
&= \sup_{\substack{a \in \Gc^{\gamma, \circ}\\ b \in \Gc(\pi^\gamma)}  } \left\{ {\rm d} \left(a, \Gc(\pi^\gamma) \right) , {\rm d} \left(b,  \mathcal G^{\gamma,\circ} \right)\right\}\\
& \le \sup_{a \in \Gc^{\gamma, \circ}  } {\rm d} \left(a, \Gc(\pi^\gamma) \right)
    + \sup_{b \in \Gc(\pi^\gamma)      } {\rm d} \left(b,  \mathcal G^{\gamma,\circ} \right) \\
& \le \sup_{a \in \Gc^{\gamma, \circ} \cap K_\varepsilon^\square} {\rm d} \left(a, \Gc(\pi^\gamma) \right)
    + \sup_{a \in \Gc^{\gamma, \circ} \cap J_\varepsilon^\square} {\rm d} \left(a, \Gc(\pi^\gamma) \right) \\
& 
    + \sup_{b \in \Gc(\pi^\gamma) \cap K_\varepsilon^\square} {\rm d} \left( b,  \mathcal G^{\gamma,\circ} \right)
    + \sup_{b \in \Gc(\pi^\gamma) \cap J_\varepsilon^\square} {\rm d} \left( b,  \mathcal G^{\gamma,\circ} \right)  \ .
\end{split}
\end{equation}
On $K_\varepsilon^\square$, we are dealing with graphs of functions, so that
\begin{align*}
\sup_{a \in \Gc^{\gamma, \circ}  } {\rm d} \left(a, \Gc(\pi^\gamma) \right) &= \sup_{t \in K_\varepsilon} {\rm d} \left( \xb_t + (\pi_t^\gamma -\xb_t) \,  e^{-D_t^\gamma}, \Gc (\pi^\gamma)\right) \\
&\le \sup_{t \in K_\varepsilon} \left|  \xb_t + \left( \pi_t^\gamma - \xb_t \right) e^{-D_t^\gamma} - \pi^\gamma_t \right|
\end{align*}
where the inequality comes from a ``slice by slice'' bound. The same argument gives that
\begin{equation*}
\sup_{b \in \Gc(\pi^\gamma) \cap K_\varepsilon^\square} {\rm d} \left( b,  \mathcal G^{\gamma,\circ} \right) \le \sup_{t \in K_\varepsilon} \left|  \xb_t + \left( \pi_t^\gamma - \xb_t \right) e^{-D_t^\gamma} - \pi^\gamma_t \right| \ .
\end{equation*}
Recalling Eq. \eqref{eq:bobound22}, we get
\begin{equation*}
\begin{split}
  {\rm d}_\H ( \mathcal G^{\gamma,\circ}, \mathcal G\left( \pi^\gamma \right) )
& \le 2 \sup_{t \in K_\varepsilon} \left|  \xb_t + \left( \pi_t^\gamma - \xb_t \right) e^{-D_t^\gamma} - \pi^\gamma_t \right|
    \\
    &+ \sup_{a \in \Gc^{\gamma, \circ} \cap J_\varepsilon^\square} {\rm d} \left(a, \Gc(\pi^\gamma) \right)     + \sup_{b \in \Gc(\pi^\gamma) \cap J_\varepsilon^\square} {\rm d} \left( b,  \mathcal G^{\gamma,\circ} \right) \\
    & \le 2 \sup_{t \in K_\varepsilon}  \left\{ \left|  \pi_t^\gamma - \xb_t \right|  \big( 1 - e^{-D_t^\gamma} \big)  \right\} \\
&+ \sup_{a \in \Gc^{\gamma, \circ} \cap J_\varepsilon^\square} {\rm d} \left(a, \Gc(\pi^\gamma) \right) 
    + \sup_{b \in \Gc(\pi^\gamma) \cap J_\varepsilon^\square} {\rm d} \left( b,  \mathcal G^{\gamma,\circ} \right) \\
    &   \le 2   \sup_{t \in K_\varepsilon} \left\{ \left|  \pi_t^\gamma - \xb_t \right| \big( 1 - e^{-D_t^\gamma} \big) \right\} \\
& + \sup_{a \in J_\varepsilon^\square} {\rm d} \left(a, \Gc(\pi^\gamma) \right) 
    + \sup_{b \in J_\varepsilon^\square} {\rm d} \left( b,  \mathcal G^{\gamma,\circ} \right) \ .
\end{split}
\end{equation*}
Now, any point in $J_\varepsilon^\square$ is at most at distance $\varepsilon$ of $J_0^\square = \sqcup_i \{ J_i \} \times [0,1]$. This gives:
\begin{equation*}
\begin{split}
   {\rm d}_\H ( \mathcal G^{\gamma,\circ}, \mathcal G\left( \pi^\gamma \right) )
& \le 2 \varepsilon
    + \sup_{a \in J_0^\square } {\rm d} \left(a, \Gc(\pi^\gamma) \right) 
    + \sup_{b \in J_0^\square } {\rm d} \left( b,  \mathcal G^{\gamma,\circ} \right) \\
&   + 2  \sup_{t \in K_\varepsilon} \left\{ \left|  \pi_t^\gamma - \xb_t \right|  \big( 1 - e^{-D_t^\gamma} \big) \right\}  \ .
\end{split}
\end{equation*}
Because $\mathcal G^{\gamma,\circ}$ contains the set of vertical lines $J_0^\square$, $\sup_{b \in J_0^\square } {\rm d} \left( b,  \mathcal G^{\gamma,\circ} \right) = 0$. Regarding the term $\sup_{a \in J_0^\square } {\rm d} \left(a, \Gc(\pi^\gamma) \right) $, we know that
\begin{equation}
\forall \ 1\le j \le L, \quad \exists s, s' \in [J_j-\varepsilon, J_j+ \varepsilon] \quad \text{\rm{ s.t. }} \quad \pi_{s}^{\gamma} \le \eta_\gamma, \quad 1-\pi_{s'}^{\gamma} \le \eta_\gamma
\end{equation}
with $\lim_{\gamma \to \infty} \eta_\gamma =0$. This claim can be proved exactly as for Eq. \eqref{eq:IVT}, replacing $\pi^{\delta_\gamma, \gamma}$ by the simpler process $\pi^\gamma$ during the proof. Since $\pi^\gamma$ has continuous trajectories, the Intermediate Value Theorem implies that
\begin{equation*}
\lim_{\gamma \to \infty} \sup_{a \in J_0^\square } {\rm d} \left(a, \Gc(\pi^\gamma) \right) =0 \ .
\end{equation*}
Therefore, by sending $\varepsilon$ to $0$ after sending $\gamma$ to infinity, Eq. \eqref{eq:onemore} is proved once we show that
\begin{equation}
\label{eq:HCERES}
\begin{split}
\limsup_{\varepsilon \to 0} \limsup_{\gamma \to \infty} \  \sup_{t \in K_\varepsilon} \left\{ \left|  \pi_t^\gamma - \xb_t \right| \big( 1 - e^{-D_t^\gamma} \big) \right\} =0 \ .
\end{split}
\end{equation}
Observe now that
\begin{align*}
&   \sup_{t \in K_\varepsilon} \left\{ \left| \pi_t^\gamma - \xb_t \right|  \big( 1- e^{-D_t^\gamma} \big) \right\} \\
\leq &  \sup_{j \le N}  \left\{ \sup_{t \in K_\varepsilon \cap [S_j, T_j]} \left\{  \left| \pi_t^\gamma - \xb_t \right|  \big( 1- e^{-D_t^\gamma} \big) \right\}
       +
       \sup_{t \notin \sqcup_k [S_k, T_k] } \left| \pi_t^\gamma - \xb_t \right| \right\}  \\
\leq &  \sup_{j \in I_\varepsilon^\gamma} \left\{ \sup_{t \in K_\varepsilon \cap [S_j, T_j]} \big( 1- e^{-D_t^\gamma} \big) 
      \right\} 
      +
       \sup_{t \notin \sqcup_k [S_k, T_k] } \left| \pi_t^\gamma - \xb_t \right| \ .
\end{align*}
Thanks to Lemma \ref{lemma:reduction}, we find on the good event, that
\begin{align*}
     & \limsup_{\varepsilon \to 0} \limsup_{\gamma \to \infty} \,  \sup_{t \in K_\varepsilon} \left\{ \left|  \pi_t^\gamma - \xb_t \right| \big( 1 - e^{-D_t^\gamma} \big) \right\}   \\
\leq &  \limsup_{\varepsilon \to 0} \limsup_{\gamma \to \infty} \   \sup_{j \in I_\varepsilon^\gamma}\;   \sup_{t \in K_\varepsilon \cap [S_j, T_j]} \big( 1- e^{-D_t^\gamma} \big)  \ .
\end{align*}
In the last line we used also Eq. \eqref{eq:Iepsilongamma}. Because of the inequality $1-e^{-x} \le x$ for all $x\ge 0$, Eq. \eqref{eq:HCERES} follows from assumption \eqref{eq:no_damping} and the proof is complete.

%
%

\bigskip
\noindent
{\bf Case 2: $C>2$.} Here $\Gc^\infty =\Gc (\xb)$. The proof in this case is slightly different. By Eq. \eqref{eq:hausdu}, we have that 
\begin{align*}
{\rm d}_\H( \Gc^{\gamma,\circ}, \Gc (\xb))
& \le {\rm d}_\H\left( 
   \Gc^{\gamma,\circ} \cap J_{\varepsilon}^\square, 
   \Gc (\xb) \cap J_{\varepsilon}^\square
   \right)
   +
   {\rm d}_\H\left( 
   \Gc^{\gamma,\circ} \cap K_{\varepsilon}^\square, 
  \Gc (\xb) \cap K_{\varepsilon}^\square
   \right) \ .
\end{align*}
Since the graphs $\Gc (\xb)$ and $\Gc^{\gamma, \circ}$ contain both the set $J_0^\square$ of vertical lines, we get easily that 
\begin{equation*}
 \limsup_{\varepsilon \to 0} \limsup_{\gamma \to \infty} {\rm d}_\H\left( 
   \Gc^{\gamma,\circ} \cap J_{\varepsilon}^\square, J_{\varepsilon}^\square 
   \right) =0
\end{equation*}
and
\begin{equation*}
 \limsup_{\varepsilon \to 0} {\rm d}_\H\left( 
   \Gc (\xb) \cap J_{\varepsilon}^\square, J_{\varepsilon}^\square 
   \right) =0
\end{equation*}
so that it remains only to prove that
\begin{equation*}
  \limsup_{\varepsilon \to 0} \limsup_{\gamma \to \infty} 
   {\rm d}_\H\left( 
   \Gc^{\gamma,\circ} \cap K_{\varepsilon}^\square, 
  \Gc (\xb) \cap K_{\varepsilon}^\square
   \right) =0 \ .
\end{equation*}
Since, on $K_\varepsilon$, we are dealing with the graphs of functions, we give a bound ``slice by slice'':
\begin{align*}
   & 
   {\rm d}_\H\left( 
   \Gc^{\gamma,\circ} \cap K_{\varepsilon}^\square, 
   \Gc( \xb ) \cap K_{\varepsilon}^\square
   \right)\\
\leq & \sup_{t \in K_\varepsilon} \left| \xb_t + \left( \pi_t^\gamma - \xb_t \right) e^{-D_t^\gamma}
       - \xb_t \right| \\
\leq & \sup_{t \in K_\varepsilon} \left| \pi_t^\gamma - \xb_t \right| e^{-D_t^\gamma} \\
\leq & \sup_{j\le L} \left\{ \sup_{t \in K_\varepsilon \cap [S_j, T_j]} e^{-D_t^\gamma}
   \    + \
       \sup_{t \notin \sqcup [S_j, T_j] } \left| \pi_t^\gamma - \xb_t \right| \right\} \\
\leq & \exp\left( -\inf_{j\le L}  \inf_{t \in K_\varepsilon \cap [S_j, T_j]} D_t^\gamma \right) 
    \    + \ 
       \sup_{j \le L} \sup_{t \notin \sqcup [S_j, T_j] } \, \left| \pi_t^\gamma - \xb_t \right| \ .
\end{align*}
Thanks to Lemma \ref{lemma:reduction}, we find on the good event $\Omega'_\eta$ (recall Eq. \eqref{eq:Omegaprime}):
\begin{align*}
   & 
   {\rm d}_\H\left( 
   \Gc^{\gamma,\circ} \cap K_{\varepsilon}^\square, 
   \Gc( \pi^\gamma) \cap K_{\varepsilon}^\square
   \right)\\
\leq & \ \varepsilon
       +
       \exp\left( -\inf_{j \le L}  \inf_{t \in K_\varepsilon \cap [S_j, T_j]} D_t^\gamma \right) \\
= & \ \varepsilon
       +
       \exp\left( -\inf_{j \in I_\varepsilon^\gamma} \inf_{t \in K_\varepsilon \cap [S_j, T_j]} D_t^\gamma \right) \ .
\end{align*}
Notice in the last equality the appearance of the index set $I_\varepsilon^\gamma$ because if $j \notin I_\varepsilon^\gamma$ then $K_\varepsilon \cap [S_j, T_j] =\emptyset$ (see the definition in Eq. \eqref{eq:Iepsilongamma}). This latter bound goes to zero by assumption \eqref{eq:damping}.
\end{proof}

\subsection{Coordinates via logistic regression}
\label{subsection:logit}

In the previous paper \cite{bernardin2018spiking}, a crucial role was played by the scale function which is the unique change of variable $f = h_\gamma$ such that  $f(\pi_t^\gamma)$ is a martingale. This uniqueness is of course up to affine transformations. Another useful change of variable is as follows. Instead of asking for a vanishing drift, one can ask for a constant volatility term. 

\begin{lemma}
Recall Eq. \eqref{eq:sde_repeated}. Up to affine transformations, the unique function $f:(0,1)\to \mathbb R$ such that $f(\pi^\gamma)$ has constant volatility term is the logistic function
$$ f: x\in (0,1) \to f(x):=\log \tfrac{x}{1-x} \ .$$
The process $Y_t^\gamma := f( \pi_t^\gamma )$ satisfies:
\begin{align}
   dY_t^\gamma
   = &  \label{eq:champ1}
   \sqrt{\gamma} dW_t
   +
   \left( \lambda (2p-1)
           + \lambda p e^{-Y_t^\gamma}
           -
           \lambda (1 - p) e^{Y_t^\gamma}
           +
           \dfrac\gamma 2 \tanh \left(\tfrac{Y_t^\gamma}{2} \right)
     \right) dt \\
   = &   \label{eq:champ2}
   \sqrt{\gamma} dB_t
   +
     \left( \lambda (2p-1)
           + \lambda p e^{-Y_t^\gamma}
           -
           \lambda (1 - p) e^{Y_t^\gamma}
           + \gamma \left( \xb_t - \half \right)
     \right) dt  \ 
\end{align}
where $(W_t \; ; \; t\ge 0)$ and $(B_t \; ; \; t\ge 0)$ are related by Eq. \eqref{eq:innovation}.
\end{lemma}

\begin{proof}

This elementary lemma is proved in Appendix \ref{app:logistic}. 

\end{proof}

Given this information the instantaneous damping term in Eq. \eqref{eq:a_u} takes a particularly convenient expression:
\begin{align}
\label{eq:damping_Y}
    a_t^\gamma
= & \lambda p \ e^{Y_t^\gamma} \ 
    + \ 
    \lambda (1-p) \ e^{-Y_t^\gamma} \ .
\end{align}

\bigskip

{\bf Informal discussion:} In particular, in order to prove that the damping term $D^\gamma$ in Eq. \eqref{eq:dampingterm} either converges to zero or diverges to infinity, it suffices to control for $t \in [\delta_\gamma, H]$:
$$
   D^\gamma_t = \int_{t-\delta_\gamma}^t a_u^\gamma \ du
   \asymp
   \int_{t-\delta_\gamma} ^t \left( e^{Y^\gamma_u} + e^{-Y^\gamma_u} \right) \ du \ .
$$
Depending on whether $\pi^\gamma_u \approx 0$ ($Y^\gamma_u \approx -\infty$) or $\pi^\gamma_u \approx 1$ ($Y^\gamma_u \approx \infty$), one of the two expressions in the integrand is dominant. Assuming $\pi^\gamma_u \approx 0$ on the entire interval $[t-\delta_\gamma,t]$, we have:
$$
   \int_{t-\delta_\gamma}^t a_u^\gamma \ du
   \asymp
   \int_{t-\delta_\gamma}^t e^{-Y^\gamma_u} \ du \ .
$$
Continuing:
\begin{align*}
   Y^\gamma_u - Y^\gamma_s
   = & 
   \sqrt{\gamma} \left( W_u - W_s \right)
   + 
   \int_s^u 
   \left( \frac{\lambda p}{\pi^\gamma_v}
           -
           \frac{\lambda (1 - p)}{1-\pi^\gamma_v}
           +
           \half \gamma \left( 2\pi^\gamma_v - 1 \right)
     \right) du \\
   = & 
   \sqrt{\gamma} \left( W_u - W_s \right)
   - 
   \half \gamma(u-s)
   +
   \int_s^u 
   \left( \frac{\lambda p}{\pi^\gamma_v}
           -
           \frac{\lambda (1 - p)}{1-\pi^\gamma_v}
           +
           \gamma \pi^\gamma_v
     \right) dv \ .
\end{align*}
We have thus proved the expression which is useful for $\pi^\gamma_u \approx 0$:
\begin{align}
\label{eq:expression_inc_Y}
   Y^\gamma_u - Y^\gamma_s
   = & 
   \sqrt{\gamma} \left( W_u - W_s \right)
   - 
   \half \gamma(u-s)
   +
   \int_s^u 
   \left( \frac{\lambda p}{\pi^\gamma_v}
           -
           \frac{\lambda (1 - p)}{1-\pi^\gamma_v}
           +
           \gamma \pi^\gamma_v
     \right) dv \ .
\end{align}
If $\pi^\gamma_u \approx 1$, then the useful expression is:
\begin{align}
\label{eq:expression_inc_Y2}
   Y^\gamma_u - Y^\gamma_s
   = & 
   \sqrt{\gamma} \left( W_u - W_s \right)
   + 
   \half \gamma(u-s)
   +
   \int_s^u 
   \left( \frac{\lambda p}{\pi^\gamma_v}
           -
           \frac{\lambda (1 - p)}{1-\pi^\gamma_v}
           -
           \gamma (1-\pi^\gamma_v)
     \right) dv \ .
\end{align}

\subsection{Path transforms} In order to systematically control the fluctuations of the process $Y^\gamma$, we make the following change of variables. For $t \geq 0$, define:
\begin{align}
\label{def:a}
\alpha^\gamma_t := & \ Y^\gamma_t - \log \left( \lambda p \right)  \ ,\\
\label{def:b}
\beta^\gamma_t := & \ \sqrt{\gamma}\  W_t - \frac{\gamma}{2} t + r^\gamma_t \ ,\\
\label{def:r}
r^\gamma_t := & \ \int_0^t \left[ \lambda(2p-1) - \lambda (1-p) e^{Y^\gamma_u} 
+ \frac{\gamma}{2}\left( 1 + \tanh\left(\tfrac{Y^\gamma_u}{2}\right) \right) 
\right]\, du \ .
\end{align}
The choice of letter for $r^\gamma$ is that it will later play the role of a residual quantity. Thanks to this reformulation, the SDE defining $Y$ in Eq. \eqref{eq:champ1} becomes:
\begin{align}
\label{eq:ode_ab}
d \alpha^\gamma_t = & \ d\beta^\gamma_t + e^{-\alpha^\gamma_t} dt \ .
\end{align}
The following lemma gives two ways of integrating Eq. \eqref{eq:ode_ab} -- in the sense that we consider $\beta^\gamma$ known, $\alpha^\gamma$ unknown and vice versa. 

\begin{lemma}
\label{lemma:path_transforms}
Consider two real-valued semi-martingales $\alpha$ and $\beta$ satisfying Eq. \eqref{eq:ode_ab}. Then for all $0 \leq s \leq t$, we have the forward and backward formulas:
$$
\alpha_{s,t} = \beta_{s,t} + \log\left( 1 + \frac{1}{e^{\alpha_s}} \int_s^t e^{-\beta_{s,u}} du \right) \ ,
$$
$$
\alpha_{s,t} = \beta_{s,t} - \log\left( 1 - \frac{1}{e^{\alpha_t}} \int_s^t e^{\beta_{u,t}} du \right) \ .
$$

\end{lemma}
\begin{proof}
See Appendix \ref{app:pathtransform}.
\end{proof}

\subsection{Controlling the residual \texorpdfstring{$r^\gamma$}{r} } Recall that in the context of Proposition \ref{proposition:more_reduction_to_damping}, we need to control:
$$ D_t^\gamma = \int_{t-\delta_\gamma}^t a^\gamma_u \ du $$
for $t\in [S_j^\gamma (\varepsilon), T_j^\gamma (\varepsilon)]$ where $j \in  I_\varepsilon^\gamma$, i.e. the $j$'s corresponding to a spike in the limit. 

%

\medskip

Examining this specific interval $[S^\gamma_j (\varepsilon), T^\gamma_j (\varepsilon)]$ with $j\in I_\varepsilon^\gamma$, it corresponds to one of the following two situations\\

\begin{enumerate}[i)]
\item $\pi_{S^\gamma_j (\varepsilon)}^\gamma =1-\varepsilon, \quad  \pi_{T^\gamma_j (\varepsilon)}^\gamma = 1- \varepsilon/2$: \quad  in the limit, it is a spike from $1$ to $1$ for $\mathbb X$\ ; \label{eq:spike11}\\
\item $\pi_{S_j^\gamma (\varepsilon)}^\gamma =\varepsilon, \quad \pi_{T^\gamma_j (\varepsilon)}^\gamma = \varepsilon/2$: \quad  in the limit, it is a spike from $0$ to $0$ for $\mathbb X$\ . \label{eq:spike00}\\
\end{enumerate}
By symmetry, we only have to consider case ii).

\begin{lemma}
\label{lemma:residual_control}
Recall that $r^\gamma_{s,t}$ denotes the increment of $r$ defined by Eq. \eqref{def:r}. Fix two arbitrary positive constants $A$ and $B$. Then, for all $j \in I_\varepsilon^\gamma$ corresponding to a spike from $1$ to $1$, i.e. like in \eqref{eq:spike11}, or a spike from $0$ to $0$, i.e. like in \eqref{eq:spike00}, the following holds
\begin{align}
\sup_{S^\gamma_j (\varepsilon) - A\delta_\gamma \leq s \leq t \leq T^\gamma_j (\varepsilon)+B\delta_\gamma}
|r^\gamma_{s,t}| = o_\gamma^\varepsilon( \log \gamma ) \ .
\end{align}
The implied function is random, yet finite, depends on $A,B$ and $\varepsilon>0$, but is independent of $j$ and $\gamma$.
\end{lemma}
\begin{proof}
To lighten the notation we omit sometimes during the proof the parameter $\gamma$. Moreover the reader has to remember the notations defined at the end of Section \ref{subsection:further_remarks}. 

\medskip

We prove the claim only in the case \eqref{eq:spike00} of a spike from $0$ to $0$, since the other case is similar. By definition of the $S_i$'s and $T_i$'s, we have that 
\begin{equation}
\label{eq:ineqTjSj}
\forall u \in [S_j, T_j], \quad   \varepsilon/2 \le \pi_u \le 1-\varepsilon/2 \ .
\end{equation}
In fact, thanks to the separation argument, the right bound holds on a much longer interval as claimed in Eq. \eqref{eq:strenghtened_claim}. As such, recalling the definition of $\S_\varepsilon$ from Eq. \eqref{eq:S_epsilon_def} and the fact that $\lim_{\gamma \to \infty} {\rm d}_\H (\Gc (\pi^\gamma), \mathbb X)=0$, if $\gamma$ is sufficiently large to have ${\mathbb S}_\varepsilon - {\rm d}_\H (\Gc (\pi^\gamma), \mathbb X)  > \sup(A,B) \ \delta_\gamma$, then
\begin{equation}
\label{eq:ubpourpi}
\sup_{S_j-A\delta_\gamma \le u \le  T_j+B\delta_\gamma} \pi_u \leq 1 - \varepsilon /2
\end{equation}
which implies
\begin{equation}
\label{eq:ubpourY}
 \sup_{S_j-A\delta_\gamma \le u \le T_j+B\delta_\gamma} Y_u \le \log(2/\varepsilon) \ .
\end{equation}

Again within the range $S_j - A\delta_\gamma \leq s \leq u \leq t \leq T_j + B\delta_\gamma$, let us control the process $r$ from Eq. \eqref{def:r}. We have:
\begin{align*}
     & \ \left| r_{s,t} \right| \\
\leq & \ 
\left[ \lambda(2p-1) + \lambda(1-p) \exp\left( \sup_{S_j-A\delta_\gamma \le u \le T_j +B \delta_\gamma} Y_u \right) \right] (t-s)
+
\frac{\gamma}{2} \int_s^{t} \left( 1 + \tanh\left(\tfrac{Y_u}{2}\right) \right) du \\
\leq & \ \Oc^\varepsilon_\gamma (1) \left( \delta_\gamma + T_j - S_j \right)
       + {\gamma}\int_s^{t} \pi_u du \ ,
\end{align*}
where we used in the last line the fact that
$$ 
   1 + \tanh\left(\tfrac{Y_u}{2}\right)
 = 1 + \frac{e^{Y_u}-1}{e^{Y_u}+1}
 = \frac{2}{1+e^{-Y_u}}
 = 2 \pi_u \ .
$$
By Corollary 2.4 in \cite{bernardin2018spiking} we know that the time spent by $\pi$ in the interval $[\half \varepsilon, 1-\half \varepsilon]$ during the time window $[0,H]$ is of order $\Oc_\gamma^\varepsilon(1/\gamma)$. Hence $\sum_j |T_j-S_j| = \Oc_\gamma^\varepsilon(1/\gamma)$. For any $\eta_\gamma$ going to $0$ as $\gamma$ goes to infinity, we have that:
\begin{align*}
\left| r_{s,t} \right| \leq & \ 
     \ \Oc_\gamma^\varepsilon(1) \left( \delta_\gamma + \frac{1}{\gamma} \right)
       + {\gamma} \int_s^{t} \pi_u du \\
\leq & \ 
     \ \Oc_\gamma^\varepsilon(1) \delta_\gamma
       + {\gamma} \int_s^{t} \pi_u \mathds{1}_{\{ \pi_u < \eta_\gamma \}}\, du
       + {\gamma} \int_s^{t} \pi_u \mathds{1}_{\{ \eta_\gamma \leq \pi_u \leq 1 - \half \varepsilon\}} \, du\\
\leq &  \ \Oc_\gamma^\varepsilon(1) \delta_\gamma
       + {\eta \gamma} \left( t - s\right) 
       + {\gamma} \int_s^{t} \pi_u \mathds{1}_{\{ \eta_\gamma \leq \pi_u \leq 1 - \half \varepsilon\}} \, du\\
\leq & \ \Oc_\gamma^\varepsilon(1) \delta_\gamma
       + \eta \gamma \Oc_\gamma^\varepsilon(1) \left( \delta_\gamma + \frac{1}{\gamma} \right) 
       + {\gamma} \int_s^{t} \pi_u \mathds{1}_{\{ \eta_\gamma \leq \pi_u \leq 1 - \half \varepsilon\}}\,  du\\
\leq & \ \Oc_\gamma^\varepsilon(1) \delta_\gamma
       + \Oc_\gamma^\varepsilon(1) \eta \log \gamma 
       + {\gamma} \int_s^{t} \pi^\gamma_u \mathds{1}_{\{ \eta_\gamma \leq \pi_u \leq 1 - \half \varepsilon\}} \, du \ .
\end{align*}

In order to control the last term in the above equation, we need to refine the previously invoked Corollary 2.4 in \cite{bernardin2018spiking}. This is done in Lemma \ref{lemma:additive_functional2} and we have then that
\begin{align*}  
\gamma \int_s^t \pi_u \mathds{1}_{\left\{ \eta_\gamma  \leq \pi_u \leq 1-\half \varepsilon \right\}}\, du
= & \ \Oc_\gamma^\varepsilon(\log |{\eta}_\gamma|) \ .
\end{align*}
As such, we can take $\eta_\gamma = \log \log \gamma/\log \gamma$ in order to have:
\begin{align*}
\left| r_{s,t} \right| \leq & \ 
     \ \Oc_\gamma^\varepsilon( \delta_\gamma ) 
       + \Oc_\gamma^\varepsilon( \log |\eta_\gamma| ) 
       + \Oc_\gamma^\varepsilon( \eta_\gamma \log \gamma ) \\
= & \ \Oc_\gamma^\varepsilon( \log \log \gamma ) \\
=  & \ o_\gamma^\varepsilon( \log \gamma ) \ .
\end{align*}
\end{proof}

\subsection{Fast feedback regime \texorpdfstring{$C<2$}{C<2} }
\label{subsection:fast_feedback}

To lighten the notation we omit usually in the sequel the parameter $\gamma$. Moreover the reader has to remember the notations defined at the end of Section \ref{subsection:further_remarks}. 

\medskip

As announced in Proposition \ref{proposition:more_reduction_to_damping}, we only need to prove a uniform absence of damping:
\begin{align}
\lim_{\gamma \rightarrow \infty}
\sup_{j \in I_\varepsilon^\gamma}
\sup_{t \in [S_j, T_j]}
D_t^\gamma
= 0 \ . 
\end{align}

\bigskip

We proceed by symmetry, as in the proof of Lemma \ref{lemma:residual_control}, by considering only spikes from $0$ to $0$. As in the proof of that lemma,  we have then that for any $t \in [S_j, T_j]$,
$$
 \sup_{t-\delta_\gamma \le u \le t} Y_u \le \Oc_\gamma^\varepsilon(1) \ .
$$
Hence:
\begin{align*}
D_t^\gamma=  \int_{t-\delta_\gamma}^t a_u \, du
&\lesssim
   \int_{t-\delta_\gamma}^t \left( e^{-Y_u} + e^{Y_u} \right) \, du \\
   &\lesssim_\varepsilon  \delta_\gamma + \int_{t-\delta_\gamma}^{t} e^{-Y_u} \, du \\
&=  \delta_\gamma + e^{-Y_{t}} \int_{t-\delta_\gamma}^t  e^{Y_t-Y_u} \, du \\
& \lesssim_\varepsilon \delta_\gamma + \int_{t-\delta_\gamma}^{t} e^{Y_{u,t}} \, du \ .  
\end{align*}
Let us now control this last term. Thanks to the reformulation of Eq. (\ref{def:a}--\ref{def:r}) and then the backward formula of Lemma \ref{lemma:path_transforms} we have:
\begin{align*}
  & \int_{t-\delta_\gamma}^{t} e^{Y_{u,t}} du \\
= & \int_{t-\delta_\gamma}^{t} e^{\alpha_{u,t}} du \\
= & \int_{t-\delta_\gamma}^{t} \frac{e^{\beta_{u,t}}}
                      {\left( 1 - \frac{1}{e^{\alpha_t}} \int_u^t e^{\beta_{v,t}} dv \right)}
                      du \\
= &  e^{\alpha_t}
     \int_{t-\delta_\gamma}^{t} \frac{ \frac{1}{e^{\alpha_t}} e^{\beta_{u,t}}}
                      {\left( 1 - \frac{1}{e^{\alpha_t}} \int_u^t e^{\beta_{v,t}} dv \right)}
                      du \\
= &  e^{\alpha_t}
     \int_{t-\delta_\gamma}^{t} \frac{d}{du}\left[  \log \left( 1 - \frac{1}{e^{\alpha_t}} \int_u^t e^{\beta_{v,t}} dv \right) \right] \ du  \\
= &  - e^{\alpha_t}
     \log \left( 1 - \frac{1}{e^{\alpha_t}} \int_{t-\delta_\gamma}^t e^{\beta_{u,t}} du \right) \ .
\end{align*}
Going back to the previous equation, we find:
\begin{align*}
  D_t^\gamma
& \lesssim_\varepsilon \delta_\gamma - e^{\alpha_t}
     \log \left( 1 - \frac{1}{e^{\alpha_t}} \int_{t-\delta_\gamma}^t e^{\beta_{u,t}} du \right) \\
& = \delta_\gamma - \frac{e^{Y_t}}{\lambda p}
     \log \left( 1 - \lambda p e^{-Y_t} \int_{t-\delta_\gamma}^t e^{r_{u,t} + \sqrt{\gamma} W_{u,t} - \frac{\gamma}{2} (t-u)} du \right) \ .
\end{align*}

Now, recall that since $t \in [S_j, T_j]$, $\pi_t$ is in the middle of a spike away from $0$ and $1$ -- see Eq. \eqref{eq:ineqTjSj}. As such $Y_t$ is bounded from below and from above. This is unlike for $u \in [t-\delta_\gamma, t]$, where $Y_u$ is bounded only from above. In any case, this yields $e^{Y_t} = \Oc^\varepsilon_\gamma(1)$ and $e^{-Y_t} = \Oc^\varepsilon_\gamma(1)$. Therefore it suffices to prove:
\begin{align}
   \label{eq:to_prove_fast_feedback}
    \sup_{S_j - \delta_\gamma \leq t \leq T_j} \int_{t-\delta_\gamma}^t 
    e^{r_{u,t} + \sqrt{\gamma} W_{u,t} - \frac{\gamma}{2} (t-u)} du
    & \stackrel{\gamma \rightarrow \infty}{\longrightarrow} 0 \ .
\end{align}
Focusing on Eq. \eqref{eq:to_prove_fast_feedback}, we have thanks to Lemma \ref{lemma:residual_control} and the change of variable $u = \frac{2 \log \gamma}{\gamma} w$:
\begin{align*}
  & \int_{t-\delta_\gamma}^t 
    e^{r_{u,t} + \sqrt{\gamma} W_{u,t} - \frac{\gamma}{2} (t-u)} du \\
= & \int_{0}^{\delta_\gamma}
    \exp\left(r_{t-u,t} + \sqrt{\gamma} W_{t-u,t} - \frac{\gamma}{2} u\right) du \\
= & \int_{0}^{\delta_\gamma}
    \exp\left( o_\gamma^\varepsilon(\log \gamma) + \sqrt{\gamma} W_{t-u,t} - \frac{\gamma}{2} u\right) du \\
= & \frac{2 \log \gamma}{\gamma}
    \gamma^{o_\gamma^\varepsilon(1)}
    \int_{0}^{\frac{C}{2}}
    \exp\left( \sqrt{\gamma} W_{t-\frac{2 \log \gamma}{\gamma} w,t} - w \log \gamma\right) dw \ .
\end{align*}
Recall now L\'evy's modulus of continuity theorem \cite[Chapter 1, Theorem 2.7]{revuz2013continuous}. Let $\beta := \left( \beta_t \ ; \ t \geq 0 \right)_{t\ge 0}$ be any fixed standard one dimensional Brownian motion and define
\begin{equation}
\label{eq:modulcont}
\begin{split}
&\omega_\beta (h) := \sup_{t,t' \in [0,H]^2, |t-t'|\leq h} 
             \frac{ \left| \beta_{t',t} \right| }
                  { \varphi(h) } \quad \text{with} \quad \varphi(h) = \sqrt{2h \log \tfrac{1}{h} } \ , \\
& \text{then a.s. (and therefore in probability): } \quad \lim_{h' \to 0} \ \sup_{0<h<h'}\omega_\beta (h) = 1 \ .
\end{split}
\end{equation} 

Because of the Dambis-Dubins-Schwartz coupling, $W = W^\gamma$ actually depends on $\gamma$. Therefore, the control provided by L\'evy's modulus of continuity cannot be used in its almost sure version but only in its probability convergence version. We introduce the notation:
$$
 \omega_{\gamma, C}
 :=
 \sup_{0\leq z \leq \frac{C}{2}} \omega_W \left(  \tfrac{2 \log \gamma}{\gamma} z  \right) 
$$
and we have that 
$$\omega_{\gamma, C}= 1 + o_{\P, \gamma}(1) $$
where $o_{\P, \gamma}(1) $ denotes a random variable converging to $0$ in probability as $\gamma$ goes to infinity. This is because for any $\delta>0$, we have that 
\begin{equation*}
\mathbb P \left( \left\vert \omega_{\gamma, C} - 1 \right\vert \ge \delta\right)= \mathbb P \left( \left\vert \sup_{0\leq z \leq \frac{C}{2}} \omega_\beta \left(  \tfrac{2 \log \gamma}{\gamma} z  \right) -1 \right\vert \ge \delta \right) \xrightarrow[{\gamma \to \infty}]{} 0 \ ,
\end{equation*}
the first equality holding because $W=W^\gamma$ and $\beta$ have the same law.

Going back to the proof of Eq. \eqref{eq:to_prove_fast_feedback}, we deduce by Lemma \ref{lemma:residual_control} that:
\begin{align*}
  & \int_{t-\delta_\gamma}^t 
    e^{r_{u,t} + \sqrt{\gamma} W_{u,t} - \frac{\gamma}{2} (t-u)} du \\
= & \frac{2 \log \gamma}{\gamma}
    \gamma^{o_\gamma^\varepsilon(1)}
    \int_{0}^{\frac{C}{2}}
    \exp\left( \sqrt{\gamma} W_{t-\frac{2 \log \gamma}{\gamma} w,t} - w \log \gamma\right) dw \\
\leq & \frac{2 \log \gamma}{\gamma}
    \gamma^{o_\gamma^\varepsilon(1)}
    \int_{0}^{\frac{C}{2}}
    \exp\left( \sqrt{\gamma} \ \omega_{\gamma,C}\ 
                              \varphi\left(  \tfrac{2 \log \gamma}{\gamma} w  \right) - w \log \gamma \right) dw \\
= & \frac{2 \log \gamma}{\gamma}
    \gamma^{ o_\gamma^\varepsilon(1) }
    \int_{0}^{\frac{C}{2}}
    \exp\left( 2\ \omega_{\gamma,C}
               \ \sqrt{w \log \gamma \log \tfrac{\gamma}{2 w \log \gamma}} - w \log \gamma  \right) dw \\
= & \frac{2 \log \gamma}{\gamma}
    \gamma^{ o_\gamma^\varepsilon(1) }
    \int_{0}^{\frac{C}{2}}
    \exp\left( 2\ \omega_{\gamma,C}
               \ \sqrt{w} (1+o_\gamma(1))\log \gamma - w \log \gamma  \right) dw \\
=    & \frac{2 \log \gamma}{\gamma}
    \gamma^{ o_\gamma^\varepsilon(1) }
    \int_{0}^{\frac{C}{2}}
    \exp\left( (\omega_{\gamma,C} -1 )
               2 \sqrt{w} (1+o_\gamma (1))\log \gamma \right)
     \gamma^{2 \sqrt{w} (1+o_\gamma (1)) - w } dw \\
\leq & \frac{2 \log \gamma}{\gamma}
    \gamma^{ o_\gamma^\varepsilon(1) }
    \exp\left( \left| \omega_{\gamma, C} -1 \right|
               \sqrt{2C} (1+o_\gamma(1))\log \gamma \right)
     \int_{0}^{\frac{C}{2}}
    \gamma^{2 \sqrt{w} (1+o(1)) - w } dw \\
=    & 2 \gamma^{ o_\gamma^\varepsilon(1) 
                + \left| -1 + \omega_{\gamma, C} \right|
               \sqrt{2C} (1+o_\gamma (1)) }
     \int_{0}^{\frac{C}{2}}
    \gamma^{ - (1-\sqrt{w})^2 } dw \\
 \leq & C \gamma^{ o_\gamma^\varepsilon(1) + o_{\mathbb P,\gamma}(1) } \gamma^{-(1-\sqrt{C/2})^2}
 \ .
\end{align*}

Observe that this upper bound goes to zero as $\gamma \rightarrow \infty$ for any $C<2$. Therefore we are done. We have proved Eq. \eqref{eq:no_damping}, which indeed gives no damping.


%

\subsection{Slow feedback regime \texorpdfstring{$C>8$}{C>2} }
\label{subsection:slow_feedback}

To lighten the notation we omit sometimes in the sequel the parameter $\gamma$. Moreover the reader has to remember the notations defined at the end of Section \ref{subsection:further_remarks}. 

\medskip

As announced in Proposition \ref{proposition:more_reduction_to_damping}, we only need to prove there is damping:
\begin{align}
\label{eq:inf_D007}
\lim_{\gamma \rightarrow \infty}
\inf_{j \in I_\varepsilon^\gamma}
\inf_{t \in [S_j, T_j]}
D_t^\gamma
= \infty \ . 
\end{align}

\bigskip

By Corollary 2.4 in \cite{bernardin2018spiking} we know that the time spent by $\pi$ in the interval $[\varepsilon/2, 1-\varepsilon/2]$ during the time window $[0,H]$ is of order $\Oc_\gamma^\varepsilon(1/\gamma)$. Hence $\sum_j |T_j-S_j| = \Oc_\gamma^\varepsilon(1/\gamma)$. Therefore, for $\gamma$ large enough, for any $j \in \{0, \ldots, N_\varepsilon^\gamma\}$, 
\begin{equation*}
T_j -\delta_\gamma \le S_j \ .
\end{equation*}
Hence, in the above infimum defined by Eq. \eqref{eq:inf_D007}, since for any $j\in I_\varepsilon^\gamma$ and $t\in [S_j,T_j]$, 
$$[T_j-\delta_\gamma, S_j] \subset  [t-\delta_\gamma, t]\ ,$$
and $a \ge 0$, we are reduced to prove that (recall Eq. \eqref{eq:dampingterm}):
$$
   \lim_{\gamma \to \infty} \int_{T_j-\delta_\gamma}^{S_j} a_u \,  du = \infty
   \ , 
$$
uniformly in $j \in I_\varepsilon^\gamma$,  i.e. for the $j$'s corresponding to a spike (recall the definition \eqref{eq:Iepsilongamma}).

\medskip
 As we have seen before, examining any interval $[S_j, T_j]$, $j \in I_\varepsilon^\gamma$, corresponds to one of the following two situations:\\
\begin{enumerate}[i)]
\item $\pi_{S_j} =1-\varepsilon, \, \pi_{T_j} = 1- \varepsilon/2$: in the limit, it is a spike from $1$ to $1$ for $\mathbb X$\ ; \label{eq:spike111}
\item $\pi_{S_j} =\varepsilon, \, \pi_{T_j}  = \varepsilon/2$: in the limit, it is a spike from $0$ to $0$ for $\mathbb X$\ . \label{eq:spike000}\\
\end{enumerate}
By symmetry, we only have to consider case \eqref{eq:spike000}. Since, by Eq. \eqref{eq:damping_Y}, 
\begin{equation}
\label{eq:int-eqcge2}
\int_{T_j-\delta_\gamma}^{S_j} a_u  \,  du \ge \lambda (1-p) \int_{T_j-\delta_\gamma}^{S_j} e^{-Y_u} \,   du 
\end{equation}
we only have to consider the limit of the integral on the right-hand side of the previous display{\footnote{In fact the neglected term in this inequality vanishes as $\gamma$ goes to infinity since $\pi$ remains at distance $\varepsilon/2$ from $1$ (hence $e^{Y}$ remains bounded) on the time interval considered, and the length of the time interval is of order $\delta_\gamma$, which goes to $0$ as $\gamma$ goes to infinity.}}.

\bigskip
\noindent
{\bf Step 1: Starting backward from $S_j$, the process $Y$ reaches $-\log \gamma + K$.} 
Here $K:= K(C, \lambda, p)$ is a large but fixed constant independent of $\gamma$ and $\varepsilon$. Let us prove that there is a random time $S_j-\tau := S_j-\tau_K \in [T_j-\delta_\gamma, S_j]$ such that:
$$ Y_{S_j-\tau} = - \log \gamma + K \ .$$
Without loss of generality, we can look for $S_j-\tau \in [S_j - \delta^\prime_\gamma, S_j]$ with $\delta^\prime_\gamma=C' \tfrac{\log \gamma}{\gamma}$ and  $8 < C'<C$. Indeed,  we will have then,  for $\gamma$ large enough, that  $T_j -\delta_\gamma=T_j - C \frac{\log \gamma}{\gamma} < S_j - C' \frac{\log \gamma}{\gamma}=S_j-\delta^\prime_\gamma$ because $\sum_j |T_j-S_j| = \Oc_\gamma^\varepsilon(1/\gamma)$. For the sake of notational simplicity, the new constant $C'$ and $\delta_\gamma'$ will be denoted $C$ and $\delta_\gamma$.

Now, by Eq. \eqref{eq:ode_ab}, for any $u<v$, we have:
\begin{align}
\label{eq:ode_ab_explicit} 
   Y_{u,v}
   = &
   \sqrt{\gamma}\  W_{u, v}
   - 
   \half \gamma (v-u)
   +
   r_{u,v}
   +
   \lambda p
   \int_{u}^v
   e^{-Y_w} \ dw \ .
\end{align}

Fix $K>0$. Let us call $S_j -\tau$ the backward hitting time of $-\log \gamma + K$ by $Y$ starting from $S_j$. Here $\tau_K:=\tau_K^\gamma \in[0,+\infty]$ is defined by
\begin{equation*}
\tau_K := \inf\{ t \le S_j \; ; \; Y_{S_j -t} = -\log \gamma + K \} \ .
\end{equation*}
Observe that $S_j-\tau_K$ is not a stopping time. Also, we can define a random variable $w_K:=w_K^\gamma \in [0,+\infty]$ by writing
\begin{equation*}
\tau_K = w_K \cfrac{2\log \gamma}{\gamma} .
\end{equation*}
At this point, we do not even know that $\tau_K$ or $w_K$ are bounded. By convention, $\tau_K = +\infty$ if the backward hitting time is never reached.

Starting from Eq. \eqref{eq:ode_ab_explicit}, take now $v= u+w \frac{2 \log \gamma}{\gamma}$ with $w \in (0, w_K]$ and  $[u,v] \subset [S_j -\tau_K, S_j]$, and write:
\begin{align*}
Y_{u,v}
\leq & \sqrt{\gamma} W_{u,v} - \half \gamma (v-u) + r_{u,v} + \lambda p e^{-K + \log \gamma} (v-u) \\
\leq & \sqrt{\gamma} W_{u,v} - w \log \gamma + r_{u,v} + 2 \lambda p e^{-K} w \log \gamma \ .
\end{align*}
Then divide by $\log \gamma$ and invoke L\'evy's modulus of continuity theorem (see Eq. \eqref{eq:modulcont}) to find that for all $w$ in $[0,w_K]$ and $[u,v] \subset [S_j -\tau_K, S_j]$:
\begin{align*}
\frac{Y_{u,v}}{\log \gamma} 
& \leq \frac{r_{u,v}}{\log \gamma} + 2 \sqrt{w} - w (1- 2 \lambda p e^{-K}) + o_{\gamma, \P}(1) \\
& \leq \frac{r_{u,v}}{\log \gamma} + 1 - (\sqrt{w} - 1)^2 + w 2 \lambda p e^{-K} + o_{\gamma, \P}(1) \ .
\end{align*}
By shifting from $K$ to $K+\log 2 \lambda p$, there is no loss of generality in writing
\begin{align}
   \label{eq:Yuv_upper_bound}
   \frac{Y_{u,v}}{\log \gamma} & 
   \leq \frac{r_{u,v}}{\log \gamma} + 1 - (\sqrt{w} - 1)^2 + w e^{-K} + o_{\gamma, \P}(1)\ .
\end{align}
This will allow our subsequent reasoning to use absolute constants. The reasoning is decomposed into two steps. The first Step 1.1 uses the idea that on a segment large enough, the drift term in Eq. \eqref{eq:ode_ab_explicit} is the main term, overpowering the oscillation of Brownian motion, so that $-\log \gamma + K$ is reached. This yields a bound on $\tau_K$, or equivalently $w_K$. The second Step 1.2 uses this estimate and refines it in order to pinpoint the location of $w_K$ around $1$ (for $\gamma$ and $K$ large enough).

\bigskip

{\underline{Step 1.1: The initial estimate: $w^\gamma_K \leq 6$ for $K=100$.}}\\
Note that the following statements are trivially equivalent: $\tau_K \leq 12 \frac{\log \gamma}{\gamma}$, or $w_K \leq 6$, or $Y$ reaches $-\log \gamma + K$ in the time interval $[S_j - 12 \frac{\log \gamma}{\gamma}, S_j]$. 

Let us start by proving, by contraposition, that with high probability, all (or any of) these statements hold. As such, we start by supposing the converse, i.e. $w_K>6$. 

Thanks to Eq. \eqref{eq:Yuv_upper_bound} and Lemma \ref{lemma:residual_control}, used to bound $r_{u,v}$, we have that for all $w$ in $(0,6]$ and $[u,v]  \subset [S_j - 12 \frac{\log \gamma}{\gamma}, S_j]$
\begin{align}
\label{eq:bound_Y_uv}
\frac{Y_{u,v}}{\log \gamma}
\leq o_{\gamma, \P}^\varepsilon(1) + 1 - (\sqrt{w} - 1)^2 + w e^{-K}  \ .
\end{align}
Note that for $ u \in [S_j -\tau_K, S_j]$, we have
$$-\log \gamma +K \leq Y_u \leq \Oc_\gamma^\varepsilon(1) \ ,  $$
which combined with Eq. \eqref{eq:ubpourY}, implies
 $-1 + o_\gamma^K(1) \leq \frac{ Y_{u} }{ \log \gamma }\leq o_\gamma^\varepsilon(1)$. This way the smallest possible LHS in Eq. \eqref{eq:bound_Y_uv} is $-1 + o_\gamma^{\varepsilon,K} (1)$. As such, we find a contradiction as soon as there is a $w \in (0,6]$ such that:
$$
-1
> o_{\gamma, \P}^{\varepsilon, K}(1) + 1 - (\sqrt{w} - 1)^2 + w e^{-K}  \ .
$$
This is rearranged as:
$$
(\sqrt{w}-1)^2
> 2 + w e^{-K} + o_{\gamma, \P}^{\varepsilon, K}(1) .
$$
Therefore, we have a contradiction for $w>(1+\sqrt{2})^2 \approx 5.8284$ and $K = K(w)$ large enough. For example $w=6$ and $K=100$.

\bigskip

{\underline{Step 1.2: Pinpointing the location of $w_K^\gamma$: 
$\lim_{K\to \infty} \limsup_{\gamma \to \infty} \left| w_K^\gamma - 1 \right| \ = \ 0$
}}\\
Thanks to the previous estimate, we now know that for $K$ large enough ($K \geq 100$ after shift by $\log 2\lambda p$), the segment $[S_j - \tau_K, S_j]$ has length $\Oc_\gamma( \frac{\log \gamma}{\gamma} )$. We can thus safely apply Lemma \ref{lemma:residual_control} to control the term $r_{u,v}$ in Eq. \eqref{eq:Yuv_upper_bound} for all $[u,v] \subset [S_j - \tau_K, S_j]$. This yields that for all $w$ in $(0,w_K]$ and $[u,v] \subset [S_j -\tau_K, S_j]$:
\begin{align}
\label{eq:Yuv_upper_bound_refined}
\frac{Y_{u,v}}{\log \gamma} & 
\leq o_{\gamma, \P}^\varepsilon(1) + 1 - (\sqrt{w} - 1)^2 + w e^{-K}  \ .
\end{align}
Choose now $v=S_j$, $u=S_j-\tau_K$ and therefore $w=w_K$. We have then 
\begin{equation*}
\cfrac{\log(\varepsilon/(1-\varepsilon))}{\log \gamma} + 1 - \frac{K}{\log \gamma} \leq o_{\gamma, \P}^\varepsilon(1) + 1 - (\sqrt{w_K}  - 1)^2 + w_K e^{-K} \ .
\end{equation*}
This implies
\begin{equation*}
   \left( \sqrt{w_K} - 1 \right)^2 
   \leq o_{\gamma, \P}^{\varepsilon, K}(1) + w_K e^{-K} 
   \leq o_{\gamma, \P}^{\varepsilon, K}(1) + 6 e^{-K} 
   \ .
\end{equation*}
Hence
\begin{equation*}
   \left| \sqrt{w_K} - 1 \right|
   \leq \sqrt{o_{\gamma, \P}^{\varepsilon, K}(1) + 6 e^{-K}}
   \ .
\end{equation*}
Multiplying by $\sqrt{w_K} + 1 \leq 4$, we find
\begin{equation*}
   \left| w_K  - 1 \right|
   \leq 4 \sqrt{o_{\gamma, \P}^{\varepsilon, K}(1) + 6 e^{-K}}
   \ .
\end{equation*}
    
We are now done with Step 1.2. In particular, it proves that $w_K < C/2$ for $K = K(C)$ and $\gamma$ large enough, which also finishes proving Step 1.

\bigskip

\begin{rmk} 
\label{rmk:Reda}
From the backward formula of Lemma \ref{lemma:path_transforms}, we see that:
\begin{align*}
  & Y_{t-u, t}\\
= & \sqrt{\gamma} W_{t-u, t} - \frac{\gamma}{2} u 
    -
    \log \left( 1 - \lambda p e^{-Y_t} \int_{t-u}^t e^{r_{v,t} + \sqrt{\gamma} W_{v,t} - \frac{\gamma}{2} (t-v)} dv \right) \\
= & \sqrt{\gamma} W_{t-u, t} - \frac{\gamma}{2} u 
    -
    \log \left( 1 - \lambda p e^{-Y_t} \int_{0}^u e^{r_{t-v,t} + \sqrt{\gamma} W_{t-v,t} - \frac{\gamma}{2} v} dv \right) \ .
\end{align*}

From the forward formula, on the other hand:
\begin{align*}
  & Y_{t, t+u}\\
= & \sqrt{\gamma} W_{t, t+u} - \frac{\gamma}{2} u 
    +
    \log \left( 1 + \lambda p e^{-Y_t} \int_{t}^{t+u} e^{-r_{t,v} - \sqrt{\gamma} W_{t,v} + \frac{\gamma}{2} (v-t)} dv \right) \\
= & \sqrt{\gamma} W_{t, t+u} - \frac{\gamma}{2} u 
    +
    \log \left( 1 + \lambda p e^{-Y_t} \int_{0}^{u} e^{-r_{t,t+v} - \sqrt{\gamma} W_{t,t+v} + \frac{\gamma}{2} v} dv \right) \ ,
\end{align*}
which easier to control in order to prove a divergence to $\infty$.

As such, we plan on using a forward estimate once we have reached level $-\log \gamma + K$. Let us record the expression for further use:
\begin{align}
   \label{eq:Y_forward}
   Y_{t, t+u}
 = & \sqrt{\gamma} W_{t, t+u} - \frac{\gamma}{2} u 
     +
     \log \left( 1 + \lambda p e^{-Y_t} \int_{0}^{u} e^{-r_{t,t+v} - \sqrt{\gamma} W_{t,t+v} + \frac{\gamma}{2} v} dv \right) \ .
 \end{align}
 
 \end{rmk}

\bigskip
\noindent
{\bf Step 2: Conclusion.} Now, we know that $Y$ reaches $-\log \gamma + K$ at $S_j - \tau_K \in [S_j-\delta_\gamma, S_j]$ thanks to the threshold of $C>2$. Moreover, by Step 1.2, the gap between $S_j-\delta_\gamma$ and $S_j- \tau_K$ is sufficiently large as it is equal to $(\tfrac{C}{2}-w_K) \tfrac{2 \log \gamma}{\gamma} > \zeta \tfrac{\log \gamma}{\gamma}$ for some small random variable $\zeta$.

Now, because of the reasoning at the beginning of Step 1, $T_j-\delta_\gamma < S_j-\delta_\gamma'$, with $\delta_\gamma' = C' \frac{\log \gamma}{\gamma}$, $2 < C' < C$, it suffices to prove
$$
   \lim_{\gamma \to \infty} \  \int_{S_j-\delta_\gamma'}^{S_j} e^{-Y_u} \ du
   = \infty \ .
$$
To simplify notations we denote $\delta'_\gamma$ by $\delta_\gamma$ and $C'$ by $C$. Let us rearrange Eq. \eqref{eq:ode_ab_explicit} as
\begin{align*}
     & \lambda p \int_{S_j-\delta_\gamma}^{S_j} e^{-Y_u}\\
   = & 
   Y_{S_j-\delta_\gamma, S_j}
   -
   \sqrt{\gamma} \  W_{S_j-\delta_\gamma, S_j}
   + 
   \half \gamma \delta_\gamma
   -r_{S_j-\delta_\gamma, S_j} \\
   = & Y_{S_j-\delta_\gamma, S_j-\tau_K}  + Y_{S_j-\tau_K, S_j} 
       - \sqrt{\gamma} W_{S_j-\delta_\gamma, S_j} + \half C \log \gamma + o_\gamma^\varepsilon(\log \gamma)\\
   = & Y_{S_j-\delta_\gamma, S_j-\tau_K}  + (Y_{S_j} + \log \gamma - K) 
     - \sqrt{\gamma} W_{S_j-\delta_\gamma, S_j}  + \half C \log \gamma +o_\gamma^\varepsilon(\log \gamma) \ .
 \end{align*}     
   By using Levy's modulus of continuity for Brownian motion, we get 
  \begin{align*}  
    & \lambda p \int_{S_j-\delta_\gamma}^{S_j} e^{-Y_u} du\\
    \geq & Y_{S_j-\delta_\gamma, S_j-\tau_K}  - \sqrt{\gamma} (1+o_{\gamma, \mathbb P} (1) ) \sqrt{2 \delta_\gamma \log \gamma}
          + \left( \frac{C}{2} + 1 \right)\log \gamma + o_\gamma^\varepsilon(\log \gamma) \\
   \geq & Y_{S_j-\delta_\gamma, S_j-\tau_K}
        + (\tfrac{C}{2} + 1 - \sqrt{2C} )\log \gamma + o_\gamma^\varepsilon(\log \gamma) \\
   = & Y_{S_j-\delta_\gamma, S_j-\tau_K}
        + \left( \sqrt{\tfrac{C}{2}} - 1 \right)^2 \log \gamma + o_\gamma^\varepsilon(\log \gamma) \ .
\end{align*}

Now, we use a pretty loose lower bound. As in the proof of Lemma \ref{lemma:residual_control}, because spikes are separated, we know that $Y_{S_j-\delta_\gamma} \leq \Oc_\gamma^\varepsilon(1)$. This is more precisely given in Eq. \eqref{eq:ubpourY}. As such
$$
Y_{S_j-\delta_\gamma, S_j-\tau_K} 
=
-\log \gamma + K - Y_{S_j-\delta_\gamma}
\geq -(1+o^\varepsilon_\gamma (1))\log \gamma \ .
$$ 
Reinjecting this inequality in the previous lower bound yields
\begin{align*}
  \lambda p \int_{S_j-\delta_\gamma}^{S_j} e^{-Y_u}du
 \geq  -\log \gamma
      + \left( \sqrt{\frac{C}{2}} - 1 \right)^2 \log \gamma + o_\gamma^\varepsilon(\log \gamma) \ .
\end{align*}
This lower bound goes to infinity as $\gamma$ goes to infinity for 
$$
-1 + \left( \sqrt{\tfrac{C}{2}} - 1 \right)^2 > 0 \ .
$$
This equivalent to $ C > 8 $. We are done in this regime.

\subsection{Intuitions on the slow feedback regime \texorpdfstring{$C>2$}{C>2} }
\label{subsection:rmk:Cge2}

In this section we explain why we conjecture that the conclusions of the slow feedback regime proved for $C>8$ should in principle hold for $C>2$ -- even if a rigorous argument eludes us for now. Observe first that Proposition \ref{proposition:more_reduction_to_damping} and the previous Step 1 are valid for $C>2$.  Hence assuming only $C>2$, we know that $Y$ reaches $-\log \gamma + K$ at $S_j-\tau_K \in [S_j-\delta_\gamma, S_j]$ as soon as $C>2$. Moreover, we have seen that $S_j -(S_j -\tau_K)=\tau_K \gtrsim \tfrac{2 \log \gamma}{\gamma}$ is sufficient. 

Now let $\tau$ be {\it any} time such that $S_j-\tau  \in [S_j-\delta_\gamma, S_j]$ and $Y_{S_j-\tau} \leq -\log \gamma + K$. Recall Proposition \ref{proposition:more_reduction_to_damping} and Eq. \eqref{eq:int-eqcge2}. It thus suffices to find a $\tau \geq \tau_K$ such that
$$
   \lim_{\gamma \to \infty} \int_{S_j-\tau}^{S_j} e^{-Y_u} \, du  = \infty  \ .
$$
Let $0\le u \le v$. Recall Eq. \eqref{eq:ode_ab_explicit}:
\begin{align}
   Y_{u,v} = & \sqrt{\gamma}\  W_{u, v}-\half \gamma (v-u)+ r_{u,v}+ \lambda p \int_{u}^v e^{-Y_w} \ dw  \ ,
\end{align}
and the forward integration of Lemma \ref{lemma:path_transforms}, which gives (see Remark \ref{rmk:Reda} )
\begin{align}  
 \label{eq:Y_forward2}
   Y_{u, v}
 = & \sqrt{\gamma}\ W_{u, v} - \frac{\gamma}{2} (v-u) 
     +
     \log \left( 1 + \lambda p e^{-Y_u} \int_{0}^{v-u} e^{-r_{u,u+w} - \sqrt{\gamma} W_{u,u+w} + \frac{\gamma}{2} w} \, dw \right) \ .
 \end{align}
 Combining the two last expressions we obtain
 \begin{equation*}
 \lambda p \int_{u}^v e^{-Y_w} \ dw =   \log \left( 1 + \lambda p e^{-Y_u} \int_{0}^{v-u} e^{-r_{u,u+w} - \sqrt{\gamma} W_{u,u+w} + \frac{\gamma}{2} w} \, dw \right) \ .
\end{equation*}
We specialise now this expression to $u=S_j -\tau, v=S_j$ to get that 
\begin{equation}
\begin{split}
\lambda p  \int_{S_j-\tau}^{S_j} e^{-Y_u} \, du & =  \log \left( 1 + \lambda p e^{-Y_{S_j -\tau}} \int_{0}^{\tau} e^{-r_{S_j-\tau,S_j-\tau+w} - \sqrt{\gamma} W_{S_j- \tau,S_j-\tau+w} + \frac{\gamma}{2} w} \, dw \right) \\
 &= \log\left( 1 + \lambda p e^{-K} \gamma \int_{0}^{\tau} e^{ o_\gamma^\varepsilon(\log \gamma) - \sqrt{\gamma} W_{S_j-\tau, S_j-\tau+w} + \frac{\gamma}{2} w } \, dw \right)\\
 & \ge  \log\left( 1 + \lambda p e^{-K} \gamma \tau \, e^{ \tau^{-1} \int_0^\tau  \big[ o_\gamma^\varepsilon(\log \gamma) - \sqrt{\gamma} W_{S_j-\tau, S_j-\tau+w} + \frac{\gamma}{2} w \big]  \, dw } \right) \ .
\end{split}
\end{equation}
Here, for the second equality we used Lemma \ref{lemma:residual_control} and for the last inequality Jensen inequality. Using $\log(1+x) \ge \log x$ for $x>0$ we get 
 \begin{equation*}
\lambda p  \int_{S_j-\tau}^{S_j} e^{-Y_u} \, du \gtrsim  {\mathcal O}^K_\gamma (1) + \log (\gamma \tau) + o^{\varepsilon}_\gamma (\log \gamma) \, +\,  \cfrac{\gamma \tau}{4} \,  - \, \frac{1}{\tau} \int_0^\tau \sqrt{\gamma} \, W_{S_j-\tau, S_j-\tau+w} \, dw  \ .
\end{equation*}
Since $\frac{\log \gamma}{\gamma} \lesssim_{K, \varepsilon} \, \tau \,  \lesssim_{K,\varepsilon}  \tfrac{\log \gamma}{\gamma}$, we are done if we can prove that
\begin{equation*}
\frac{1}{\tau} \int_0^\tau \sqrt{\gamma} \, W_{S_j-\tau, S_j-\tau+w} \, dw = o_\gamma^{\varepsilon, K} (\log \gamma)\ .
\end{equation*}

Recall that we have some freedom in the choice of the constant $K$ and that $\tau=\tau_K$ depends on $K$. Then, if we could find $K$ such that $S_j-\tau_K$ is a generic point for Brownian motion, i.e. where the Law of Iterated Logarithm \cite[Chapter 2, Theorem 1.9]{revuz2013continuous} is satisfied, instead of the full L\'evy modulus of continuity \cite[Chapter 1, Theorem 2.7]{revuz2013continuous}, we could conclude the proof for $C>2$.

\bigskip 
\noindent


\bigskip
\section{Acknowledgements}
The authors are grateful to R. Ch\'etrite for useful initial discussions. C.P is supported by the ANR project `ESQuisses', grant number  ANR-20-CE47-0014-01, the ANR project `Quantum Trajectories',  grant  number ANR-20-CE40-0024-01 and the program `Investissements d'Avenir'  ANR-11-LABX-0040 of the French National Research Agency. C. P. is also  supported by the ANR project Q-COAST ANR-19-CE48-0003. This work has been also supported by the 80 prime project StronQU of MITI-CNRS:  `Strong noise limit of stochastic processes and application of quantum 
systems out of equilibrium'.

\appendix

\section{Scale function and time change \texorpdfstring{\cite{bernardin2018spiking}}{[BCC+22]}
}
\label{app:scale function}
Let us recall the expressions of the scale function and change of time used in the paper \cite{bernardin2018spiking}. We define the scale function $h = h_\gamma$ of Eq. \eqref{eq:sde} as:
\begin{equation}
\label{eq:expressionh_gamma}
h_\gamma(x) =
   x_0
   +
   \int_{x_0}^x
   dy \ 
   \exp\left[ \frac{2\lambda}{\gamma} g(y) \right]
   \ ,
   \end{equation}
   where
   \begin{equation}
\label{eq:expressiong}
   g(y) := 
   p \left( \frac{1}{y} + \log \frac{1-y}{y} \right)
   +
   (1-p)\left( \frac{1}{1-y} + \log \frac{y}{1-y} \right)
   \ .   
\end{equation}
%

Thanks to the Dambis-Dubins-Schwartz theorem, if $\pi^\gamma$ denotes the solution of Eq. \eqref{eq:sde}, there is a Brownian motion $\beta$ starting from $x_0 \in (0,1)$ such that:
\begin{equation}
\label{eq:DDS-equation}
 \pi_t^\gamma = h_\gamma^{\langle -1 \rangle}\left( \beta_{T_t^\gamma} \right) \ .
\end{equation}
The time change $T^\gamma$ is given by:
\begin{align}
\label{eq:TTgamma0}
dT_t^\gamma = & \gamma \  h_\gamma'(\pi_t^\gamma)^2\  \left[ \pi_t^\gamma (1-\pi_t^\gamma) \right] ^2 \ dt =  \gamma \ e^{\frac{4\lambda}{\gamma} g(\pi_t^\gamma)} \  \left[ \pi_t^\gamma (1-\pi_t^\gamma) \right]^2 \ dt \ 
\end{align}
and the inverse change is given by \cite[Subsection 3.2]{bernardin2018spiking}
\begin{align*}
d \left[T^\gamma_\ell\right]^{\langle-1 \rangle} =  \
              \frac{1}{\gamma \ 
              \left[ h_\gamma' \circ h_\gamma^{\langle-1 \rangle}(\beta_\ell)\right]^2 \
              h_\gamma^{\langle-1 \rangle}(\beta_\ell)^2 \
              (1-h_\gamma^{\langle-1 \rangle}(\beta_\ell))^2} \ d\ell  =:  \ \varphi_\gamma(\beta_\ell) d\ell \ .
\end{align*}
For $s\le t$ and $y\in {\mathbb R}$ we denote $L_{s,t}^y (\beta)$ the occupation time of level $y$ by $\beta$ during the time interval $(s,t)$. Via the occupation time formula:
\begin{align}
\label{eq:ttgammaL}
\left[T^\gamma_\ell\right]^{\langle-1 \rangle} = & \ \int_0^\ell \ \varphi_\gamma(\beta_u) \ du
            =: \ \int_\R \varphi_\gamma(a) \ L_\ell^a(\beta) \ da \ 
\end{align}
and the weak convergence of $\varphi_\gamma$ to the mixture $(2\lambda {p})^{-1} \delta_0 + (2\lambda (1-{p}))^{-1} \delta_1$, we can deduce the almost sure convergence:
\begin{align}
\label{eq:cv_T_inv}
\left[T^\gamma_\ell\right]^{\langle-1 \rangle}_{\ell} \; \stackrel{\gamma \rightarrow \infty}{\longrightarrow} 
   \; & 
   \frac{1}{2\lambda {p}} L^0_\ell(\beta) 
 + \frac{1}{2\lambda (1-{p})} L^1_\ell(\beta) 
   \ ,
\end{align}
uniformly on all compact sets of the form $[0,L]$. 

\bigskip

We observe finally that introducing 
\begin{equation*}
\sigma_t
   :=
   \inf \left\{ \ell \geq 0, \ 
   \frac{L_\ell^0(\beta)}{2\lambda{p}} + \frac{L_\ell^1(\beta)}{2 \lambda (1-{p})} > t \right\} \ ,
\end{equation*}
we have that 
\begin{equation*}
(\xb_t)_{t\ge 0} =(\beta_{\sigma_t})_{t\ge 0} \ ,
\end{equation*}
the equality being in law.

\section{Asymptotic analysis of a singular additive functional}
\label{section:analysis}

Throughout the paper, it is important to control the damping term
$$ \int_s^t a^\gamma_u \ du
 = \int_s^t a( \pi^\gamma_u ) \ du \ ,$$
 where $a^\gamma_u$ is given by Eq. \eqref{eq:a_u}.
 More generally, for any positive map $f: [0,1] \rightarrow \R_+$, we define the additive functional:
\begin{align}
\label{eq:additive_functional_A}
A^\gamma_{s,t}(f) := \int_s^t f( \pi^\gamma_u ) \, du \ .
\end{align}

\begin{lemma}
\label{lemma:additive_functional}
We have the exact expression:
$$
  A^\gamma_{s,t}(f)
= \int_0^1  \ 
  \frac{ f(x) }{\gamma  x^2 (1-x)^2 } e^{-\frac{2\lambda}{\gamma} g(x)} L^{h_\gamma(x)}_{T_s^\gamma, T_t^\gamma}(\beta) \ dx \ .
$$
In particular, with $f={\mathds 1}$ we get that 
\begin{equation}
\label{eq:A1}
A_{s,t}^\gamma ({\mathds 1}) =\frac{1}{\gamma} \int_0^1 \cfrac{e^{-\tfrac{2\lambda}{\gamma} g(x) }}{x^2 (1-x)^2} \  L_{T_s^\gamma, T_t^\gamma}^{h_\gamma (x)} (\beta) \ dx \ = t-s \ .
\end{equation}

\end{lemma}
\begin{proof}
Recalling Eq. \eqref{eq:TTgamma0} and Eq. \eqref{eq:DDS-equation} we have that
\begin{align*}
  & A^\gamma_{s,t}(f) = \int_s^t  \ f( \pi^\gamma_u )  \, du =  \int_s^t   \  \frac{f( \pi^\gamma_u )}{\gamma \exp\left[ \frac{4\lambda}{\gamma} g(\pi^\gamma_u) \right] \left[\pi^\gamma_u (1-\pi^\gamma_u)\right]^2 } \, dT_u^\gamma\\
= & \int_s^t \frac{f \circ h_\gamma^{\langle -1 \rangle}( \beta_{T_u^\gamma} ) }
                  {\gamma \exp\left[ \frac{4\lambda}{\gamma} g \circ h_\gamma^{\langle -1 \rangle}( \beta_{T_u^\gamma} ) \right]
             h_\gamma^{\langle -1 \rangle}( \beta_{T_u^\gamma} )^2 (1-h_\gamma^{\langle -1 \rangle}( \beta_{T_u^\gamma} ))^2 } \, dT_u^\gamma \\
= & \int_{T_s^\gamma}^{T_t^\gamma}
             \frac{f \circ h_\gamma^{\langle -1 \rangle}( \beta_{\ell} ) }
                  {\gamma \exp\left[ \frac{4\lambda}{\gamma} g \circ h_\gamma^{\langle -1 \rangle}( \beta_{\ell} ) \right]
             h_\gamma^{\langle -1 \rangle}( \beta_{\ell} )^2 (1-h_\gamma^{\langle -1 \rangle}( \beta_{\ell} ))^2 }\,  d\ell \ .
\end{align*}
Invoking the occupation time formula:
\begin{align*}
 A^\gamma_{s,t}(f) 
= & \int_{\R} \ 
             \frac{f \circ h_\gamma^{\langle -1 \rangle}( y ) }
                  {\gamma \left( h_\gamma' \circ h_\gamma^{\langle -1 \rangle}( y ) \right)^2
             h_\gamma^{\langle -1 \rangle}( y )^2 (1-h_\gamma^{\langle -1 \rangle}( y ))^2 }
             L^y_{T_s^\gamma, T_t^\gamma}(\beta) \, dy  \\
= & \int_0^1 \ 
             \frac{f(x) }
                  {\gamma  h_\gamma'(x)
             x^2 (1-x)^2 }
	L^{h_\gamma(x)}_{T_s^\gamma, T_t^\gamma}(\beta) \, dx  \\
= & \int_0^1 \ 
             \frac{ f(x) }
                  {\gamma
             x^2 (1-x)^2 } e^{-\frac{2\lambda}{\gamma} g(x)} L^{h_\gamma(x)}_{T_s^\gamma, T_t^\gamma}(\beta) \, dx  \ .
\end{align*}
\end{proof}

\begin{lemma}
\label{lemma:additive_functional2} 
Let $\varepsilon \in (0,1)$ be fixed and $\eta_\gamma>0$ such that $\lim_{\gamma \to \infty} \eta_\gamma =0$. Then, $\mathbb P$ a.s., for any $0 \le s \le t \le H$, we have that 
\begin{align*}  
\gamma \int_s^t \pi_u^\gamma\,  \mathds{1}_{\left\{ \eta_\gamma \leq \pi_u^\gamma \leq 1-\half \varepsilon \right\}} du
= \ \Oc_\gamma^\varepsilon(\log |{\eta_\gamma}|) \ .
\end{align*}
\end{lemma}

\begin{proof}
By the occupation time formula in Lemma \ref{lemma:additive_functional}:
\begin{align*}
    \gamma \int_s^t \pi_u^\gamma  \mathds{1}_{\left\{ \eta_\gamma \leq \pi_u^\gamma \leq 1-\half \varepsilon \right\}} \, du
= & \ \gamma \int_s^t \pi^\gamma_u \mathds{1}_{\left\{ \eta_\gamma \leq \pi_u^\gamma \leq \half \right\}} \, du
    + \gamma \int_s^t \pi_u^\gamma  \mathds{1}_{\left\{ \half \leq \pi^\gamma_u \leq 1-\half \varepsilon \right\}} \, du\\
= & \gamma \int_0^1  \ 
  \frac{ x \mathds{1}_{\left\{ \eta_\gamma \leq x \leq \half \right\}} }{\gamma  x^2 (1-x)^2 } e^{-\frac{2\lambda}{\gamma} g(x)} L^{h_\gamma(x)}_{T_s^\gamma, T_t^\gamma}(\beta) \, dx
  + \Oc_\gamma^\varepsilon(1) \\
= & \int_{\eta_\gamma}^\half  \ 
  \frac{1}{x (1-x)^2 } e^{-\frac{2\lambda}{\gamma} g(x)} L^{h_\gamma(x)}_{T_s^\gamma, T_t^\gamma}(\beta) \, dx
  + \Oc_\gamma^\varepsilon(1) \ .
 \end{align*} 
The previous $\Oc_\varepsilon (1)$ term results from Corollary 2.4 in \cite{bernardin2018spiking} which states that the time spent by $\pi$ in some fixed (i.e. independent of $\gamma$) interval $[a,b] \subset(0,1)$ is of order $1/\gamma$. 

To control the local time increment we observe that
\begin{align*}
L^{h_\gamma(x)}_{T_s^\gamma, T_t^\gamma}(\beta) \ \leq  \  \sup_{u \in [0,H]} \sup_{a \in \R} L^{a}_{T_u^\gamma}(\beta)\  \leq  \ \sup_{\ell \in [0, T_H^{\langle -1 \rangle}]} \sup_{a \in \R} L^{a}_{\ell}(\beta)\ ,
\end{align*}
and we recall that, by Eq. \eqref{eq:cv_T_inv},  $T_H^{\langle -1 \rangle}$ converges to a finite limit. As such, this is controlled by the maximal local time over a finite (random)  time interval. Hence 
\begin{align*}  
\gamma \int_s^t \pi_u^\gamma \mathds{1}_{\left\{ \eta_\gamma \leq \pi^\gamma_u \leq 1-\half \varepsilon \right\}} \ du
\lesssim &
  \int_{\eta_\gamma}^\half  \ 
  x^{-1} \ e^{-\frac{2\lambda}{\gamma} g(x)} \ dx
  + \Oc_\gamma^\varepsilon(1) \\
= & \ \Oc_\gamma (\log|{\eta_\gamma}|) + \Oc_\gamma^\varepsilon(1) \\
= & \ \Oc_\gamma^\varepsilon(\log |{\eta}_\gamma|) \ .
\end{align*}

\end{proof}

\section{Few technical lemmas}

\subsection{Logistic transform. Proof of Lemma \ref{subsection:logit} }
\label{app:logistic}

 In order to lighten notation we omit the superscript $\gamma$ in the next equations. For a given smooth function $f$, It\^o formula yields
\begin{align*}
    df(\pi_t)
= & \left( -f'(\pi_t) \lambda (\pi_t - p)
           +
           \half \gamma f''(\pi_t) \pi_t^2 (1-\pi_t)^2
     \right) dt
  + f'(\pi_t) \pi_t (1-\pi_t) \sqrt{\gamma} dW_t \ .
\end{align*}
As such, $f(\pi_t)$ has constant volatility term, say $\sqrt{\gamma}$, if and only if:
\begin{align*}
    & 1 = f'(x)x(1-x) 
\Longleftrightarrow \ 
    f(x) = \log\frac{x}{1-x} + \log \kappa \ ,
\end{align*}
for a certain choice of constant $\kappa>0$. The first claim is proved.

Now, choosing $\kappa=1$ for convenience, let us derive the SDE for $Y=f(\pi)$. By using $ f'(x) = \tfrac{1}{x(1-x)}, \quad  f''(x) = \tfrac{2x-1}{x^2(1-x)^2} \ ,$ we obtain:
\begin{align*}
    dY_t
= & \left( -\lambda \frac{\pi_t - p}
                         {\pi_t(1-\pi_t)}
           +
           \half \gamma \left( 2\pi_t-1 \right)
     \right) dt
  + \sqrt{\gamma} dW_t \\
 = & \left( \lambda p
           + \lambda p e^{-Y_t}
           -
           \lambda (1 - p) e^{Y_t}
           -
           \lambda (1 - p)
           +
           \half \gamma \left( 2\pi_t - 1 \right)
     \right) dt
  + \sqrt{\gamma} dW_t \
  \ ,
\end{align*}
hence the first expression \eqref{eq:champ1} with $2 \pi_t - 1 = \frac{2}{1+e^{-Y_t}} - 1 = \tanh(\tfrac{Y_t}{2})$.

For the second expression \eqref{eq:champ2}, recalling that
$ dW_t = dB_t + \sqrt{\gamma} \left( \xb_t - \pi_t \right) dt \ ,$
we get:
\begin{align*}
    dY_t
 = & \left( \lambda (2p-1)
           + \lambda p e^{-Y_t}
           -
           \lambda (1 - p) e^{Y_t}
           +
           \half \gamma \left( 2\pi_t - 1 \right)
           + \gamma \left( \xb_t - \pi_t \right)
     \right) dt
  + \sqrt{\gamma} dB_t \\
 = & \left( \lambda (2p-1)
           + \lambda p e^{-Y_t}
           -
           \lambda (1 - p) e^{Y_t}
           + \gamma \left( \xb_t - \half \right)
     \right) dt
  + \sqrt{\gamma} dB_t  \ .
\end{align*}

\subsection{Path transform} 
\label{app:pathtransform}
We start by writing:
$$
\alpha_{s,t} = \beta_{s,t} + \int_s^t e^{-\alpha_u} \, du \ .
$$
\bigskip 

\noindent {\bf Backward integration:} Consider $t$ as fixed and $s$ as varying. Then differentiate in $s$ the expression for $s \leq t$:
$$
\alpha_{s,t} = \beta_{s,t} + \frac{1}{e^{\alpha_t}} \int_s^t e^{\alpha_{u,t}} du \ .
$$
It yields:
$$
\frac{d}{ds}\left(
\alpha_{s,t} - \beta_{s,t}
\right)
=
-\frac{1}{e^{\alpha_t}} e^{\alpha_{s,t}} 
=
-\frac{1}{e^{\alpha_t}} e^{\alpha_{s,t} - \beta_{s,t}} e^{\beta_{s,t}} \ .
$$
Equivalently:
$$
\frac{d}{ds}\left(
e^{-(\alpha_{s,t} - \beta_{s,t})}
\right)
=
\frac{1}{e^{\alpha_t}} e^{\beta_{s,t}} \ .
$$
Integrating on $[s,t]$ gives:
$$ 1 - e^{-(\alpha_{s,t} - \beta_{s,t})} =  \frac{1}{e^{\alpha_t}} \int_s^t e^{\beta_{u,t}} du \ , $$
which gives the backward formula.

\bigskip 

\noindent
{\bf Forward integration:} The other way around, fix $s$ and take $t$ as varying. Then differentiate in $s$ the expression for $s \leq t$:
$$
\alpha_{s,t} = \beta_{s,t} + \frac{1}{e^{\alpha_s}} \int_s^t e^{-\alpha_{s,u}} \, du \ .
$$
It yields:
$$
\frac{d}{dt}\left(
\alpha_{s,t} - \beta_{s,t}
\right)
=
\frac{1}{e^{\alpha_s}} e^{-\alpha_{s,t}} 
=
\frac{1}{e^{\alpha_s}} e^{-(\alpha_{s,t} - \beta_{s,t})} e^{-\beta_{s,t}} \ .
$$
Equivalently:
$$
\frac{d}{dt}\left(
e^{\alpha_{s,t} - \beta_{s,t}}
\right)
=
\frac{1}{e^{\alpha_s}} e^{-\beta_{s,t}} \ .
$$
Integrating between on $[s,t]$ gives:
$$ e^{\alpha_{s,t} - \beta_{s,t}} - 1 =  \frac{1}{e^{\alpha_s}} \int_s^t e^{-\beta_{s,u}} \, du \ , $$
which gives the forward formula.

\bibliographystyle{halpha}
\bibliography{biblio}

\newcommand{\etalchar}[1]{$^{#1}$}
\begin{thebibliography}{PKHM19}

\bibitem[Ass97]{assaf1997estimating}
David Assaf.
\newblock Estimating the state of a noisy continuous time markov chain when
  dynamic sampling is feasible.
\newblock {\em The Annals of Applied Probability}, 7(3):822--836, 1997.

\bibitem[Ata98]{AtarAoP98}
Rami Atar.
\newblock Exponential stability for nonlinear filtering of diffusion processes
  in a noncompact domain.
\newblock {\em Ann. Probab.}, 26(4):1552--1574, 1998.

\bibitem[AZ97a]{Atar-ZeitouniAIHP97}
Rami Atar and Ofer Zeitouni.
\newblock Exponential stability for nonlinear filtering.
\newblock {\em Ann. Inst. H. Poincar\'{e} Probab. Statist.}, 33(6):697--725,
  1997.

\bibitem[AZ97b]{Atar-ZeitouniSIAM97}
Rami Atar and Ofer Zeitouni.
\newblock Lyapunov exponents for finite state nonlinear filtering.
\newblock {\em SIAM J. Control Optim.}, 35(1):36--55, 1997.

\bibitem[AZ98]{Atar-Zeitouni-SCL98}
Rami Atar and Ofer Zeitouni.
\newblock A note on the memory length of optimal nonlinear filters.
\newblock {\em Systems Control Lett.}, 35(2):131--135, 1998.

\bibitem[Bar06]{barnsley}
Michael~Fielding Barnsley.
\newblock {\em Superfractals}.
\newblock Cambridge University Press, Cambridge, 2006.

\bibitem[BBC{\etalchar{+}}21]{benoist2021emergence}
Tristan Benoist, C{\'e}dric Bernardin, Rapha{\"e}l Chetrite, Reda Chhaibi,
  Joseph Najnudel, and Cl{\'e}ment Pellegrini.
\newblock Emergence of jumps in quantum trajectories via homogenization.
\newblock {\em Communications in Mathematical Physics}, 387(3):1821--1867,
  2021.

\bibitem[BBT16]{bauer2016zooming}
Michel Bauer, Denis Bernard, and Antoine Tilloy.
\newblock Zooming in on quantum trajectories.
\newblock {\em Journal of Physics A: Mathematical and Theoretical},
  49(10):10LT01, 2016.

\bibitem[BCC{\etalchar{+}}22]{bernardin2018spiking}
C~Bernardin, R~Chetrite, R~Chhaibi, J~Najnudel, and C~Pellegrini.
\newblock {Spiking and collapsing in large noise limits of SDE's}.
\newblock {\em arXiv preprint arXiv:1810.05629, To Appear In Annals of Applied
  Probability}, 2022.

\bibitem[Bil13]{billingsley2013convergence}
Patrick Billingsley.
\newblock {\em Convergence of probability measures}.
\newblock John Wiley \& Sons, 2013.

\bibitem[Chi06]{chigansky2006}
Pavel Chigansky.
\newblock On filtering of markov chains in strong noise.
\newblock {\em IEEE transactions on information theory}, 52(9):4267--4272,
  2006.

\bibitem[Gol00]{golubev2000}
Georgii~Ksenofontovich Golubev.
\newblock On filtering for a hidden markov chain under square performance
  criterion.
\newblock {\em Problemy Peredachi Informatsii}, 36(3):22--28, 2000.

\bibitem[Kin92]{kingman1992poisson}
John Frank~Charles Kingman.
\newblock {\em Poisson processes}, volume~3.
\newblock Clarendon Press, 1992.

\bibitem[KL92]{Khasminski-LazarevaAoS92}
Rafail~Z. Khasminskii and Betty~V. Lazareva.
\newblock On some filtration procedure for jump {M}arkov process observed in
  white {G}aussian noise.
\newblock {\em Ann. Statist.}, 20(4):2153--2160, 1992.

\bibitem[KLK{\etalchar{+}}14]{ecc_study1}
Samira Khan, Donghyuk Lee, Yoongu Kim, Alaa~R Alameldeen, Chris Wilkerson, and
  Onur Mutlu.
\newblock The efficacy of error mitigation techniques for dram retention
  failures: A comparative experimental study.
\newblock {\em ACM SIGMETRICS Performance Evaluation Review}, 42(1):519--532,
  2014.

\bibitem[KZ96]{Khasminskii-ZeitouniSPA}
Rafail Khasminskii and Ofer Zeitouni.
\newblock Asymptotic filtering for finite state {M}arkov chains.
\newblock {\em Stochastic Process. Appl.}, 63(1):1--10, 1996.

\bibitem[{Lip}01]{LS01}
{Liptser, Robert S and Shiryaev, Albert N}.
\newblock {\em Statistics of random processes: I. General theory}, volume~1.
\newblock Springer Science \& Business Media, 2001.

\bibitem[Mab09]{mabuchi2009}
Hideo Mabuchi.
\newblock Continuous quantum error correction as classical hybrid control.
\newblock {\em New Journal of Physics}, 11(10):105044, oct 2009.

\bibitem[Pic86]{PicardSIAM}
Jean Picard.
\newblock Nonlinear filtering of one-dimensional diffusions in the case of a
  high signal-to-noise ratio.
\newblock {\em SIAM J. Appl. Math.}, 46(6):1098--1125, 1986.

\bibitem[PKHM19]{ecc_study2}
Minesh Patel, Jeremie~S Kim, Hasan Hassan, and Onur Mutlu.
\newblock Understanding and modeling on-die error correction in modern dram: An
  experimental study using real devices.
\newblock In {\em 2019 49th Annual IEEE/IFIP International Conference on
  Dependable Systems and Networks (DSN)}, pages 13--25. IEEE, 2019.

\bibitem[PZ05]{Pardoux-Zeitouni05}
\'{E}tienne Pardoux and Ofer Zeitouni.
\newblock Quenched large deviations for one dimensional nonlinear filtering.
\newblock {\em SIAM J. Control Optim.}, 43(4):1272--1297, 2004/05.

\bibitem[RBA22]{Sumith-AmitSIAM22}
Anugu~Sumith Reddy, Amarjit Budhiraja, and Amit Apte.
\newblock Some large deviation asymptotics in small noise filtering problems.
\newblock {\em SIAM J. Control Optim.}, 60(1):385--409, 2022.

\bibitem[RY13]{revuz2013continuous}
Daniel Revuz and Marc Yor.
\newblock {\em Continuous martingales and Brownian motion}, volume 293.
\newblock Springer Science \& Business Media, 2013.

\bibitem[SPW09]{dram_study}
Bianca Schroeder, Eduardo Pinheiro, and Wolf-Dietrich Weber.
\newblock Dram errors in the wild: a large-scale field study.
\newblock {\em ACM SIGMETRICS Performance Evaluation Review}, 37(1):193--204,
  2009.

\bibitem[TBB15]{tilloy2015spikes}
Antoine Tilloy, Michel Bauer, and Denis Bernard.
\newblock Spikes in quantum trajectories.
\newblock {\em Phys. Rev. A}, 92(5):052111, 2015.

\bibitem[VH07]{vanhandel}
Ramon Van~Handel.
\newblock Stochastic calculus, filtering, and stochastic control.
\newblock {\em Course notes., URL http://www. princeton.
  edu/rvan/acm217/ACM217. pdf}, 14, 2007.

\bibitem[Won64]{wonham1964}
W~Murray Wonham.
\newblock Some applications of stochastic differential equations to optimal
  nonlinear filtering.
\newblock {\em Journal of the Society for Industrial and Applied Mathematics,
  Series A: Control}, 2(3):347--369, 1964.

\end{thebibliography}

\end{document}